\documentclass[11pt, letterpaper,final]{amsart}
\usepackage{amsfonts,amsmath, amsthm, amssymb, epsfig, mathtools, xcolor}
\usepackage[english]{babel}
\usepackage[utf8]{inputenc}
\usepackage[all]{xy}
\usepackage{xspace}
\usepackage{comment}
\usepackage{setspace}
\usepackage{enumerate}
\usepackage{stmaryrd}
\usepackage{tikz}
\usepackage{tikz-cd}
\usepackage[mathscr]{eucal}

\usepackage[notcite,notref]{showkeys}

	\newcommand{\Aff}{\mathrm{Aff}}

	\newcommand{\ev}{\mathrm{ev}}
	\newcommand{\id}{\mathrm{id}}
        \newcommand{\cs}{\mathrm{C}^\ast}

	\renewcommand{\restriction}{|}

\newcommand{\Ped}{\mathrm{Ped}}

	\theoremstyle{plain}
	\newtheorem{thma}{Theorem}
	\newtheorem{introtheorem}[thma]{Theorem}

	\theoremstyle{definition}
	\newtheorem{introquestion}[thma]{Question}
	
        \newtheorem{introobservation}[thma]{Observation}

	\theoremstyle{plain}
	\newtheorem{thm}{Theorem}[section]
	\newtheorem{lemma}[thm]{Lemma}
	\newtheorem{theorem}[thm]{Theorem}
	\newtheorem{proposition}[thm]{Proposition}
	\newtheorem{corollary}[thm]{Corollary}

	\theoremstyle{definition}
	\newtheorem{definition}[thm]{Definition}
	
	\newtheorem{remark}[thm]{Remark}
	
	\newtheorem{example}[thm]{Example}

	\numberwithin{equation}{section}
	\numberwithin{figure}{section}

\begin{document}
	\title{Self-adjoint traces on the Pedersen ideal of $\cs$-algebras}
    	\author{James Gabe} 
            \address[J.~Gabe]{%
    Department of Mathematics and Computer Science\\
    University of Southern Denmark\\
    Campusvej 55\\
    5230 Odense\\
    Denmark}
    \email{gabe@imada.sdu.dk}
            \author{Alistair Miller}
            \address[A.~Miller]{%
    Department of Mathematics and Computer Science\\
    University of Southern Denmark\\
    Campusvej 55\\
    5230 Odense\\
    Denmark}
    \email{mil@sdu.dk}

\thanks{This was supported by DFF grants 1054-00094B and 1026-00371B} 


\begin{abstract}
    In order to circumvent a fundamental issue when studying densely defined traces on $\cs$-algebras -- which we refer to as the Trace Question -- we initiate a systematic study of the set $T_{\mathbb R}(A)$ of self-adjoint traces on the Pedersen ideal of $A$.     
    The set $T_{\mathbb R}(A)$ is a topological vector space with a vector lattice structure, which in the unital setting reflects the Choquet simplex structure of the tracial states. We establish a form of Kadison duality for $T_{\mathbb R}(A)$ and compute $T_{\mathbb R}(A)$ for principal twisted étale groupoid $\cs$-algebras. We also answer the Trace Question positively for a large class of $\cs$-algebras.
\end{abstract}

\maketitle

\section{Introduction}

This paper adds a new perspective for studying traces on non-unital $\cs$-algebras. For unital $\cs$-algebras $A$ there is no doubt what tracial invariant to consider:  the set $T_1(A)$ of tracial states. For non-unital $\cs$-algebras, one often considers the set $T_+(A)$ of all densely defined lower semicontinuous tracial weights with the weak$^\ast$-topology coming from the Pedersen ideal $\Ped(A)$ (the minimal dense ideal in $A$). Equivalently, $T_+(A)$ can be identified with the set of positive linear tracial functionals on $\Ped(A)$. The following is the fundamental problem for using $T_+(A)$ as the primary tracial invariant, and we will refer to it as the Trace Question.

\begin{introquestion}[The Trace Question]\label{q:traces}
    Let $A$ be a $\cs$-algebra. Suppose that $f \colon T_+(A) \to \mathbb R$ is a continuous affine function with $f(0) = 0$. Is there a self-adjoint $a\in \Ped(A)$ such that $f(\tau) = \tau(a)$ for all $\tau \in T_+(A)$?
\end{introquestion}

The Trace Question turns out to have a positive answer for many $\cs$-algebras, for instance all simple $\cs$-algebras. This is arguably why $T_+(A)$ suffices (together with K-theory and their pairing) for classifying simple non-unital separable nuclear $\mathcal Z$-stable $\cs$-algebras satisfying the UCT (proved independently in \cite{EGLN, GL2, GL3, GL4} and in \cite{CGSTW2}). More generally, the class of $\cs$-algebras for which the Trace Question has an affirmative answer is closed under stable isomorphism, and contains all $\cs$-algebras with compact primitive ideal space, all $\cs$-algebras for which $T_+(A)$ has a compact base, as well as all commutative $\cs$-algebras, 
see Corollary \ref{c:traceproblem}.

However, we do not know if the Trace Question always has an affirmative answer. The key point of this paper is to show that by enriching $T_+(A)$ with (seemingly insignificant) additional structure, one obtains a better behaved invariant where the analogous Trace Question always has an affirmative answer.

The invariant -- which we denote by $T_{\mathbb R}(A)$ -- is the subspace of $\Ped(A)^\ast$ consisting of self-adjoint traces equipped with the weak$^\ast$-topology. This turns out to be the span of $T_+(A)$ (Corollary \ref{c:contlincomb}), so it might be hard to see how this contains any new information compared to $T_+(A)$. However, the following observation -- that the related Trace Question for $T_{\mathbb R}(A)$ has a positive answer -- motivates considering $T_{\mathbb R}(A)$ instead of $T_+(A)$.

\begin{introobservation}\label{o:trace}
        Let $A$ be a $\cs$-algebra. If $f \colon T_{\mathbb R}(A) \to \mathbb R$ is a continuous linear functional, then there is a self-adjoint $a\in \Ped(A)$ such that $f(\tau) = \tau(a)$ for all $\tau \in T_{\mathbb R}(A)$.
\end{introobservation}

The subtle difference between the Trace Question and Observation \ref{o:trace}, is that while any continuous affine function $f\colon T_+(A)\to \mathbb R$ which vanishes at $0$ extends uniquely to a linear functional on $T_{\mathbb R}(A)$, it is not obvious that this extension is continuous. The Trace Question has an affirmative answer exactly if such extensions are automatically continuous.

A different approach for bypassing the Trace Question was set out by Elliott--Robert--Santiago \cite{ERS} by considering the set $T(A)$ of all lower semicontinuous tracial weights (not necessarily densely defined). They consider a class $L(T(A))$ of lower semicontinuous functions $f \colon T(A) \to [0,\infty]$. It is shown in \cite[Theorem 2]{Robert} that if $A$ is stable and $f\in L(T(A))$ then there is a positive element $a\in A$ such that $f(\tau) = \tau(a)$ for all $\tau\in T(A)$.  
Note however that $T(A)$, unlike $T_+(A)$ and $T_{\mathbb R}(A)$, does not pair with $K_0(A)$ (see Proposition \ref{p:pairing} for the pairing of $K_0(A)$ and $T_{\mathbb R}(A)$).

In order to systematically study $T_{\mathbb R}(A)$ we introduce a topology on $\Ped(A)$ as follows. Since $\Ped(A)$ is the directed union of all its $\cs$-subalgebras (Remark \ref{r:directed}), we equip $\Ped(A)$ with the induced locally convex inductive limit topology. The continuous dual space $\Ped(A)^\ast$ thus consists exactly of the linear functionals on $\Ped(A)$ which are bounded when restricted to any $\cs$-subalgebra of $\Ped(A)$. In particular, all positive linear functionals on $\Ped(A)$ are continuous. Formally, we define $T_{\mathbb R}(A) \subseteq \Ped(A)^\ast$ as the space of continuous self-adjoint traces. This is a real topological vector space which is ordered by the cone $T_+(A)$ of positive traces. Moreover, by utilising Jordan decompositions of linear functionals we give a different proof (Corollary \ref{c:lattice}) of a result by Pedersen \cite[Theorem 3.1]{Pedersen-MeasureIII} that $T_{\mathbb R}(A)$ is a vector lattice.\footnote{In the case where $A$ is unital, this gives an easy proof that the set of tracial states is a Choquet simplex. See \cite{BRtraces} for another recent short proof of this result.}

We consider the continuous dual space $T_{\mathbb R}(A)^\ast$ equipped with the compact-open topology. 
In the case where $A$ is unital, $T_{\mathbb R}(A)^\ast$ can be canonically identified with $\Aff T_1(A)$ (the space of continuous real affine functions on the tracial state space $T_1(A)$), and the compact-open topology agrees with the usual norm topology. The following is a  strengthening of Observation \ref{o:trace}. In the statement, $T_{\leq 1}(A)$ denotes the set of positive traces of norm at most $1$. The following is Theorem \ref{t:PedTcont} and Corollary \ref{c:norm}.

\begin{introtheorem}
    Let $A$ be a $\cs$-algebra. The canonical map $\Ped(A)_{\mathrm{sa}} \to T_{\mathbb R}(A)^\ast$ is a quotient map with respect to the topologies described above. 

    Consequently, for every $f\in T_{\mathbb R}(A)^\ast$ and $\epsilon >0$ there is an $a\in \Ped(A)_{\mathrm{sa}}$ such that $f(\tau) = \tau(a)$ for all $\tau\in T_{\mathbb R}(A)$ and $\| a \| \leq \sup_{\tau \in T_{\leq 1}(A)} |f(\tau)| + \epsilon$.
\end{introtheorem}

For compact convex sets $K$ there is an important duality due to Kadison \cite{Kadison-duality} between $K$ and $\Aff K$, where $K$ can be recovered from $\Aff K$ as the state space. We obtain a similar duality for $T_{\mathbb R}(A)$ and $T_{\mathbb R}(A)^\ast$. While $T_{\mathbb R}(A)$ is trivially isomorphic to the space of weak$^\ast$-continuous functionals on $T_{\mathbb R}(A)^\ast$, we emphasise that we do not consider $T_{\mathbb R}(A)^\ast$ with the weak$^\ast$-topology, but with the compact-open topology (which is the quotient topology coming from $\Ped(A)_{\mathrm{sa}}$). This is Proposition \ref{p:dualiso} and Theorem \ref{t:duality}.

\begin{introtheorem}[Duality]
    Let $A$ and $B$ be $\cs$-algebras. Then $T_{\mathbb R}(A)$ is canonically isomorphic to the space of compact-open-continuous linear functionals on $T_{\mathbb R}(A)^\ast$. This induces a contravariant categorical duality, that is, a bijective correspondence between weak$^\ast$-continuous positive linear maps $T_{\mathbb R}(B) \to T_{\mathbb R}(A)$ and compact-open-continuous positive linear maps $T_{\mathbb R}(A)^\ast \to T_{\mathbb R}(B)^\ast$. 
\end{introtheorem}

Any $\ast$-homomorphism $A \to B$ induces a weak$^\ast$-continuous positive linear map $T_{\mathbb R}(B) \to T_{\mathbb R}(A)$. In fact, we may view $T_{\mathbb R}$ as a contravariant functor on the category of $\cs$-algebras. It is often convenient to consider the covariant functor $T_{\mathbb R}(-)^\ast$ instead, which by duality contains the exact same information as $T_{\mathbb R}$.

The functor $T_{\mathbb R}$ is stable in the sense that the $(1,1)$-corner inclusion $A \to A\otimes \mathcal K(\ell^2(\mathbb N))$ induces an isomorphism $T_{\mathbb R}(A\otimes \mathcal K(\ell^2(\mathbb N))) \cong T_{\mathbb R}(A)$. More generally, if $A\subseteq B$ is a full hereditary $\cs$-algebra then we obtain an isomorphism $T_{\mathbb R}(B) \cong T_{\mathbb R}(A)$ by restriction (Corollary \ref{c:Tfullher}).  The functor $T_{\mathbb R}$ is also invariant under approximate unitary equivalence of $\ast$-homomorphisms (Remark \ref{r:aue}).

For any locally compact Hausdorff space $X$, the real trace space $T_{\mathbb R}(C_0(X))$ is frequently identified\footnote{In fact, this is how real Radon measures on $X$ are defined in the Bourbaki book \cite[Chapter III]{Bourbaki-XIII}.} with the space of real Radon measures $\mathrm{Rad}_{\mathbb R}(X)$ on $X$ with the weak$^\ast$-topology coming from $C_c(X) = \Ped(C_0(X))$, via the map sending a measure $\mu$ to the trace $f \mapsto \int_X f \mathrm d\mu$.
We study traces more generally on the $\cs$-algebra of an étale groupoid $\mathcal G$, which have been considered for instance in \cite{NeshveyevKMSGroupoid, LiRenault, AfsarSims, NeshveyevStammeier, ChristensenKMSGroupoid, LiZhang}; we are especially interested in the entire structure of the trace space, notably its topology.
Consider the set $\mathrm{Rad}_{\mathbb R}^{\mathcal G}(\mathcal G^0)$ of $\mathcal G$-invariant real Radon measures on the unit space $\mathcal G^0$ equipped with the weak$^\ast$-topology coming from $C_c(\mathcal G^0)$. We show (Corollary \ref{c:groupoidtrace}) that for a twisted étale groupoid $(\mathcal G, \mathcal L)$ (not necessarily Hausdorff), every $\mu \in \mathrm{Rad}_{\mathbb R}^{\mathcal G}(\mathcal G^0)$ induces a trace $\tau_{\mu} \in T_{\mathbb R}(\cs_\lambda(\mathcal G, \mathcal L))$ which is uniquely determined by
\[
\tau_\mu(f) = \int_{\mathcal G^0} f|_{\mathcal G^0} \mathrm d\mu
\]
for compactly supported continuous\footnote{If $\mathcal G$ is not Hausdorff these need not actually be continuous, see Section \ref{s:groupoid}.} sections $f \in \Gamma_c(\mathcal G, \mathcal L)$ of the line bundle $\mathcal L \to \mathcal G$.
Moreover, when $\mathcal G$ is Hausdorff the weak$^\ast$-topology on this set of traces coming from $\Ped(\cs_\lambda(\mathcal G, \mathcal L))$ agrees with the one coming from $C_c(\mathcal G^0)$ (Lemma \ref{l:grpoidtop}). In the case where $\mathcal G$ is principal, this determines the entire space of traces (Corollary \ref{c:principal}).

\begin{introtheorem}
    Let $(\mathcal G, \mathcal L)$ be a principal twisted étale groupoid. Then the map $\mathrm{Rad}_{\mathbb R}^{\mathcal G}(\mathcal G^0) \to T_{\mathbb R}(\cs_\lambda(\mathcal G, \mathcal L))$ described above is an isomorphism of ordered topological vector spaces.
\end{introtheorem}

\subsection*{Acknowledgement} The authors would like to thank Leonel Robert for an email correspondence on the Trace Question. The first named author would also like to thank José Carrión, Jorge Castillejos, Sam Evington, Chris Schafhauser, Aaron Tikuisis, and Stuart White for enlightening conversations throughout the years about bounded traces. The authors would also like to thank the referees for their very helpful comments.

\section{Continuous traces on the Pedersen ideal}

In \cite{Pedersen-MeasureI}, Pedersen introduced an ideal -- now known as the \emph{Pedersen ideal} $\Ped(A)$ -- which is the unique minimal two-sided norm-dense ideal in a $\cs$-algebra $A$. The standard examples of Pedersen ideals are $\Ped(\mathcal K(\mathcal H))$ (for a Hilbert space $\mathcal H$) which exactly consists of the finite rank operators on $\mathcal H$; and $\Ped(C_0(X)) = C_c(X)$ for a locally compact Hausdorff space $X$, where $C_c(X)$ consists of all compactly supported continuous functions on $X$. 

The Pedersen ideal $\Ped(A)$ is constructed as follows: let $\Ped(A)_0$ be the set of all positive elements $a\in A$ for which there exist $e\in A_+$ such that $ea=a$. 
Let $\Ped(A)_+$ denote the hereditary cone generated by $\Ped(A)_0$, i.e.~the set of positive elements $a\in A$ for which there exist $a_1,\dots, a_n \in \Ped(A)_0$ such that $a \leq \sum_{j=1}^n a_j$. Then $\Ped(A)$ is the linear span of $\Ped(A)_+$. 

As seen in \cite[Section 5.6]{Pedersen-book-automorphism} it holds that $\Ped(A) \cap A_+ = \Ped(A)_+$ and that $\Ped(A)$ is the unique minimal two-sided norm-dense ideal in $A$.

We will be considering the Pedersen ideal with a locally convex structure which in general is stronger than the usual norm topology. The starting point is the following important result of Pedersen.

\begin{lemma}[{\cite[Proposition 5.6.2]{Pedersen-book-automorphism}}]\label{l:Pedher}
Let $X\subseteq \Ped(A)$ be a finite set. The hereditary $\cs$-subalgebra of $A$ generated by $X$ is contained in $\Ped(A)$. 
\end{lemma}

In particular, whenever $x\in \Ped(A)$ then the entire $\cs$-algebra $\cs(x)$ generated by $x$ is contained in $\Ped(A)$, so $\Ped(A)$ contains many $\cs$-subalgebras of $A$ even though it is itself not norm-closed in general. When $D$ is a (hereditary) $\cs$-subalgebra of $A$ such that $D \subseteq \Ped(A)$ we will say (which is slight abuse of notation) that $D \subseteq \Ped(A)$ is a (hereditary) $\cs$-subalgebra. 

We will need the following minor variation of the above lemma. 

\begin{proposition}\label{p:upwards}
Let $A$ be a $\cs$-algebra, and let $D_1, D_2,\dots, D_n \subseteq \Ped(A)$ be $\cs$-subalgebras of the Pedersen ideal. Then the hereditary $\cs$-subalgebra of $A$ generated by $D_1, D_2,\dots, D_n$ is contained in $\Ped(A)$. 
\end{proposition}
\begin{proof}
Since we may iteratively construct the final hereditary $\cs$-subalgebra by adding in one $D_k$ at a time, it suffices to consider the case $n=2$. 
Let $B_0 = \cs(D_1,D_2)$ and let $B$ be the hereditary $\cs$-subalgebra generated by $B_0$. Fix $x\in B$. By Cohen's factorisation theorem (see \cite[Theorem 4.6.4]{BrownOzawa-book-approx} for an easy proof in the case that we need it for) we have $B= B_0 A$ so we may write $x= by$ for a $b\in B_0$ and $y \in A$. Pick separable $\cs$-subalgebras $C_j \subseteq D_j$ for $j=1,2$ such that $b\in \cs(C_1,C_2)$. As $C_j$ is separable and thus $\sigma$-unital we let $h_j \in C_j$ be strictly positive. Then $h_1,h_2 \in \Ped(A)$ and by Lemma \ref{l:Pedher} the hereditary $\cs$-subalgebra $C$ generated by $\{h_1,h_2\}$ is contained in $\Ped(A)$. Since $b\in \cs(C_1,C_2) \subseteq C \subseteq \Ped(A)$, and since $\Ped(A)$ is an ideal, it follows that $x= by \in \Ped(A)$.
\end{proof}

\begin{remark}\label{r:directed}
It follows that the Pedersen ideal is the directed union of each of the following sets:
\begin{itemize}
\item[(1)] the set of $\cs$-subalgebras of $\Ped(A)$;
\item[(2)] the set of separable $\cs$-subalgebras of $\Ped(A)$;
\item[(3)] the set of hereditary $\cs$-subalgebras of $\Ped(A)$; and
\item[(4)] the set of $\sigma$-unital hereditary $\cs$-subalgebras of $\Ped(A)$.
\end{itemize}
\end{remark}

\begin{definition}\label{d:locconvex}
Let $A$ be a $\cs$-algebra. We equip the Pedersen ideal $\Ped(A)$ with the locally convex direct limit topology coming from writing $\Ped(A) = \bigcup_{D} D$ as the union of its $\cs$-subalgebras $D \subseteq \Ped(A)$. 
\end{definition}

Hence, a linear map $\rho \colon \Ped(A) \to X$ into a locally convex space $X$ is continuous if and only if the restriction $\rho|_{D}$ is continuous for every $\cs$-subalgebra $D \subseteq \Ped(A)$. 
As a special case, any positive linear functional $\Ped(A) \to \mathbb C$ is continuous even when it is not bounded.

It is easy to see that one could have chosen any of the four directed sets from Remark \ref{r:directed}  
to define the topology on $\Ped(A)$, since a map $\rho \colon D \to X$ from a $\cs$-algebra $D$ to a topological space $X$ is continuous if and only if it is continuous on all separable $\cs$-subalgebras of $D$.\footnote{In fact, if $\rho$ was not continuous we could find a sequence $a_n \in D$ converging to $a\in D$ such that $\rho(a_n)$ does not converge to $\rho(a)$. Then $D_0 = \cs(a_n : n\in \mathbb N)$ is separable and $\rho|_{D_0}$ is not continuous.} When considering traces, we will for instance focus on hereditary $\cs$-subalgebras since these play well with the Jordan decomposition (Proposition \ref{p:JordanPed}). 

\begin{remark}
This way of defining continuous functionals on $\Ped(A)$ is similar to (and inspired by) the way in which one defines real or complex Radon measures on locally compact spaces $X$. For instance, as done in the Bourbaki book \cite[Chapter III]{Bourbaki-XIII}, this is done by identifying complex Radon measures on $X$ with linear functionals on $C_c(X) = \Ped(C_0(X))$ (the space of compactly supported functions) which are continuous when $C_c(X)$ is equipped with the same locally convex inductive limit topology as described above. 
\end{remark}

\begin{remark}
    In \cite{Pedersen-MeasureII}, Pedersen considers a different locally convex topology on $\Ped(A)$ which he uses to define $\cs$-integrals, a generalisation of traces. As we only consider traces in this paper, we think the topology defined in Definition \ref{d:locconvex} is more suitable and easier to work with.
\end{remark}

A functional $\tau \colon \Ped(A) \to \mathbb C$ is called \emph{tracial} (or a trace) if $\tau (xy) = \tau(yx)$ for all $x,y\in \Ped(A)$. Note that if $\tau$ is a continuous trace then so is its adjoint $\tau^\ast$ (given by $\tau^\ast(x) = \overline{\tau(x^\ast)}$), and thus any continuous trace $\tau$ on $\Ped(A)$ can be written as $\tau = \tau_1 + i \tau_2$ where $\tau_1 = \tfrac{1}{2} (\tau + \tau^\ast)$ and $\tau_2= \tfrac{1}{2i} (\tau- \tau^\ast)$ are continuous self-adjoint traces on $\Ped(A)$. 

\begin{definition}\label{d:traces}
We let $T_{\mathbb R}(A)$ denote the real vector space of continuous self-adjoint traces on $\Ped(A)$. We equip $T_{\mathbb R}(A)$ with the weak$^\ast$-topology. We consider $T_{\mathbb R}(A)$ as an ordered topological vector space with the positive cone $T_+(A)$ of positive traces.
\end{definition}

\begin{example}
    Let $X$ be a discrete topological space. Then $\Ped(C_0(X)) = C_c(X)$ and $T_{\mathbb R}(C_0(X))$ is canonically isomorphic to $\prod_{X} \mathbb R$ equipped with the product topology.

    More generally, if $(A_x)_{x\in X}$ is a family of $\cs$-algebras, then $\Ped(\bigoplus_{x\in X} A_x)$ is the algebraic direct sum $\bigoplus_{x\in X} \Ped(A_x)$ and $T_{\mathbb R}(\bigoplus_{x\in X} A_x)$ is canonically identified with $\prod_{x\in X} T_{\mathbb R}(A_x)$ equipped with the product topology.
\end{example}

Pedersen showed in \cite{Pedersen-MeasureI} that there is a one-to-one correspondence (given by restriction) between $T_+(A)$ and the set of lower semicontinuous densely defined tracial weights on $A$. This will not play a significant role in this paper, although we use it for convenience in Remark \ref{r:aue}.

We often consider $T_{\mathbb R}(A)$ as a space of functionals on the real vector space $\Ped(A)_{\mathrm{sa}}$ of self-adjoint elements in $\Ped(A)$. Each $\tau \in T_{\mathbb R}(A)$ restricts to a functional $\tau \colon \Ped(A)_{\mathrm{sa}} \to \mathbb R$ which is tracial in the sense that $\tau(x^\ast x) = \tau(xx^\ast)$ for all $x\in \Ped(A)$. It is straightforward to see that, conversely, any linear functional $\tau \colon \Ped(A)_{\mathrm{sa}} \to \mathbb R$ with this tracial property extends uniquely to a self-adjoint tracial functional $\tilde \tau$ on $\Ped(A)$. The extension $\tilde \tau$ is continuous if and only if $\tau \colon \Ped(A)_{\mathrm{sa}} \to \mathbb R$ is continuous, so we may identify the self-adjoint traces as a subspace $T_{\mathbb R}(A) \subseteq (\Ped(A)_{\mathrm{sa}})^*$, where its topology agrees with the weak$^\ast$ topology. This will be used often without reference.

One advantage of working with continuous traces on the Pedersen ideal is that there is a natural trace pairing with the K-theory group $K_0$. 

\begin{remark}\label{r:K0Ped}
Recall that because $\Ped(A)$ is dense in $A$ and is the union of its $\cs$-subalgebras, and since $K_0$ preserves direct limits we get that
\[  K_0(A)  =  \{ [p]_0 - [q]_0 \mid  n \in \mathbb N, p,q \in M_n(\Ped(A)^\sim), p-q \in M_n(\Ped(A)) \}.\]
\end{remark}
For $\tau \in T_{\mathbb R}(A)$ and $n \in \mathbb N$, let $\tau^{(n)} \colon M_n(\Ped(A)) \to \mathbb C$ denote the trace given by 
\begin{equation}
    \tau^{(n)} ((a_{i,j})_{i,j=1}^n) = \sum_{j=1}^n \tau(a_{j,j}) \qquad \textrm{for }(a_{i,j})_{i,j=1}^n \in M_n(\Ped(A)). 
\end{equation}

\begin{proposition}\label{p:pairing}
There is a unique bilinear pairing 
\[ \rho_A \colon K_0(A) \times T_{\mathbb R}(A) \to \mathbb R \] such that for any $n \in \mathbb N$, any projections $p,q \in M_n(\Ped(A)^\sim)$ with $p-q \in M_n(\Ped(A))$ and any $\tau \in T_{\mathbb R}(A)$, we have
\[ \rho_A([p]_0 - [q]_0, \tau) = \tau^{(n)}(p-q). \]
Moreover, $\rho_A$ is continuous.
\end{proposition}
\begin{proof}
Every element of $K_0(A)$ is of the form $[p]_0 - [q]_0$ as in the statement by Remark \ref{r:K0Ped}, so if $\rho_A$ is well-defined then it is unique. Continuity also follows because $K_0(A)$ carries the discrete topology and for each $p,q$ as above, the map $T_{\mathbb R}(A) \to \mathbb R \colon \tau \mapsto \tau^{(n)}(p-q)$ is a linear combination of evaluation functionals. Thus we need only show well-definedness.

Fix $\tau \in T_{\mathbb R}(A)$ and take $n, \tilde n \in \mathbb N$, $p,q \in M_n(\Ped(A)^\sim)$ and $\tilde p, \tilde q \in M_{\tilde n}(\Ped(A)^\sim)$ with $p-q \in M_n(\Ped(A))$ and $\tilde p - \tilde q \in M_{\tilde n}(\Ped(A))$ such that $[p]_0 - [q]_0 = [\tilde p]_0 - [\tilde q]_0$ and note that by embedding into the larger matrix algebra we may assume $\tilde n = n$. Pick a $\cs$-subalgebra $D_0 \subseteq \Ped(A)$ such that $p,q,\tilde p, \tilde q \in M_n(D_0^\sim)$ and $p-q, \tilde p - \tilde q \in M_n(D_0)$. Since $K_0$ preserves direct limits, there is a $\cs$-subalgebra $D \subseteq \Ped(A)$ containing $D_0$ such that $[p]_0 - [q]_0 = [\tilde p]_0 - [\tilde q]_0$ in $K_0(D)$. Thus for some $k \in \mathbb N$ we have $p \oplus \tilde q \oplus 1_k \sim q \oplus \tilde p \oplus 1_k$ in $M_{2n+k}(\tilde D)$ where $1_k$ is the unit of $M_k(\tilde D)$. The restriction $\tau |_D \colon D \to \mathbb C$ is bounded and extends to a bounded self-adjoint tracial functional $\tilde \tau \colon \tilde D \to \mathbb C$ (e.g.~by letting $\tilde \tau (1) = 0$). Hence
\begin{equation}
\tilde \tau^{(2n+k)}(p \oplus \tilde q \oplus 1_k) =  \tilde \tau^{(2n+k)}(q\oplus \tilde p \oplus 1_k)
\end{equation}
and therefore $\tau^{(n)}(p-q) = \tau^{(n)}(\tilde p - \tilde q)$.
\end{proof}

\section{Jordan decomposition}

We will show that $T_{\mathbb R}(A)$ is a vector lattice (although not a topological vector lattice). Consequently, the continuous traces on the Pedersen ideal are exactly the ones that are a linear combination of positive traces. This was originally proved by Pedersen \cite[Theorem 3.1]{Pedersen-MeasureIII}, but our approach gives a different description of the lattice structure in terms of the Jordan decomposition of traces on hereditary $\cs$-subalgebras of the Pedersen ideal, which we will use subsequently.

Recall the \emph{Jordan decomposition} of self-adjoint linear functionals on $\cs$-algebras: for any bounded self-adjoint linear functional $\rho$ on a $\cs$-algebra $D$, there is a unique pair $(\rho^+,\rho^-)$ of positive linear functionals on $D$ such that $\rho = \rho^+-\rho^-$ and $\| \rho\| = \|\rho^+\| + \| \rho^-\|$  (see for instance \cite[Theorem 3.2.5]{Pedersen-book-automorphism}).

\begin{proposition}\label{p:Jordan1}
Let $D$ be a $\cs$-algebra and let $\tau\in D^\ast_{\mathrm{sa}}$ be a bounded self-adjoint tracial functional with Jordan decomposition $(\tau^+,\tau^-)$. Then:
\begin{itemize}
    \item[$(a)$] $\tau^+$ and $\tau^-$ are tracial;
    \item[$(b)$] for every $d\in D_+$ and $\epsilon>0$ there exist $c_1,c_2\in D_+$ so that $c_1 + c_2 = d$ 
    and 
    \begin{equation}\label{eq:tau+c}
    \tau^+(c_1) < \epsilon, \qquad \tau^-(c_2) < \epsilon;
    \end{equation}
    \item[$(c)$] $\tau^+$ is the supremum of $\tau$ and $0$ in $D^\ast_{\mathrm{sa}}$. 
\end{itemize} 
\end{proposition}
\begin{proof}
$(a)$: If $u\in \widetilde{D}$ is a unitary, then $(u \tau^+ u^\ast, u \tau^- u^\ast)$ is a Jordan decomposition of $u\tau u^\ast = \tau$, so $\tau^\pm = \tau^\pm(u(-)u^\ast)$ by uniqueness. Hence $\tau^+$ and $\tau^-$ are tracial.

$(b)$: Let $\delta >0$ such that $\tfrac{1}{2} \|d\| \delta < \epsilon$. Pick $x\in D_{\mathrm{sa}}$ with $\| x\| \leq 1$ such that \begin{equation}
    \tau^+(x) - \tau^-(x) + \delta = \tau(x) + \delta > \| \tau\| = \| \tau^+ \| + \| \tau^- \| = \tau^+ (1) + \tau^-(1)
\end{equation}
where we extended $\tau^\pm$ to the unitisation of $D$ in the unique norm-preserving way. Note that $1-x$ and $1+x$ are positive and 
\begin{equation}
    \tau^+(1-x) + \tau^-(1+x) < \delta.
\end{equation}
Let $c_1 = \tfrac{1}{2} d^{1/2} (1-x) d^{1/2}$ and $c_2 = \tfrac{1}{2} d^{1/2} (1+x) d^{1/2}$ which are positive and $c_1 + c_2 = d$. As $\tau^+$ is positive, and is a trace by $(a)$, we have
\begin{equation}
    \tau^+(c_1) = \tfrac{1}{2} \tau^+ ((1-x)^{1/2} d (1-x)^{1/2}) \leq \tfrac{1}{2}\| d\| \tau^+(1-x) \leq \tfrac{1}{2} \| d\| \delta < \epsilon,
\end{equation}
and similarly $\tau^-(c_2) < \epsilon$.

$(c)$: Clearly $\tau^+ \geq \tau$ and $\tau^+ \geq 0$. Let $\rho \in D^\ast_{\mathrm{sa}}$ such that $\tau \leq \rho$ and $0\leq \rho$. Let $d\in D_+$ and $\epsilon>0$ and pick $c_1,c_2$ as in part $(b)$. Then
    \[
        \tau^+(c_1) - \epsilon \leq 0 \leq \rho(c_1), \qquad \tau^+ (c_2) - \epsilon \leq \tau^+ (c_2) - \tau^-(c_2) = \tau(c_2) \leq \rho(c_2),
    \]
    and thus 
    \[
    \tau^+(d) - 2\epsilon = \tau^+(c_1)+\tau^+(c_2) - 2\epsilon \leq \rho(c_1)+ \rho(c_2) = \rho(d).
    \]
    Hence $\tau^+(d) \leq \rho(d)$ for all $d\in D_+$ and therefore $\tau^+ \leq \rho$ in $D^\ast_{\mathrm{sa}}$. 
\end{proof}

\begin{remark}
    Proposition \ref{p:Jordan1}(c) gives an elementary proof (assuming the existence of Jordan decompositions of bounded self-adjoint functionals) that the set of bounded self-adjoint tracial functionals on a $\cs$-algebra forms a vector lattice with join  $\tau_1 \vee \tau_2 = (\tau_1 - \tau_2)^+ + \tau_2$ and meet $\tau_1 \wedge \tau_2 = \tau_1 - (\tau_1 - \tau_2)^+$. In particular, the set of tracial states on a unital $\cs$-algebra is a Choquet simplex.
\end{remark}

\begin{remark}
If $D$ is a $\cs$-algebra and $(\tau_i)$ is a bounded net of self-adjoint tracial functionals converging weak$^\ast$ to $\tau$, and if $(\tau_i^+, \tau_i^-)$ and $(\tau^+, \tau^-)$ are the respective Jordan decompositions, then it is in general not true that $\tau_i^+ \to \tau^+$ and $\tau_i^- \to \tau^-$ (even when we know that $(\tau_i^+)$ and $(\tau_i^-)$ both converge to tracial functionals). 

For example, let $D= C([-1,1])$ and $\tau_n = \mathrm{ev}_{1/n} - \mathrm{ev}_{-1/n}$ (where $\mathrm{ev}_t$ is evaluation at $t$). Then $\tau_n \xrightarrow{w^\ast} 0$ but $\tau_n^{\pm} = \mathrm{ev}_{\pm 1/n} \xrightarrow{w^\ast} \ev_0 \ne 0$. 
\end{remark}

\begin{lemma}\label{l:Jordan2}
Let $C\subseteq D$ be a hereditary $\cs$-subalgebra and let $\tau \in D^\ast_{\mathrm{sa}}$ with Jordan decomposition $(\tau^+,\tau^-)$. If $\tau$ is tracial then $(\tau^+|_C, \tau^-|_C)$ is the Jordan decomposition of $\tau|_C$. 
\end{lemma}
\begin{proof}
It suffices to show that $\| \tau|_C\| = \| \tau^+|_C\| + \| \tau^-|_C\|$. Clearly $\| \tau|_C\| \leq \| \tau^+|_C\| + \| \tau^-|_C\|$. For the other inequality, let $\epsilon >0$. Pick $d\in C$ a positive contraction (e.g.~from an approximate identity in $C$) such that 
\begin{equation}
\| \tau^+|_C\| + \| \tau^-|_C\| - \epsilon \leq \tau^+(d) + \tau^-(d).
\end{equation}
Use Proposition \ref{p:Jordan1} to find $c_1,c_2\in D_+$ such that $c_1 + c_2 = d$, $\tau^+(c_1) < \epsilon$ and $\tau^-(c_2) < \epsilon$. Then $c_1,c_2$ are positive contractions which are in $C$ since $C \subseteq D$ is hereditary. 
Let $c = c_1-c_2$ which is in the unit ball of $C$. Then
\begin{eqnarray}
\| \tau|_C\| \geq \tau(c) &=& \tau^+(c_1) + \tau^-(c_2) - \tau^+(c_2) - \tau^-(c_1) \nonumber\\
&>& \tau^+(d) + \tau^-(d) - 4 \epsilon \nonumber\\
&\geq& \| \tau^+|_C\| + \|\tau^-|_C\| - 5 \epsilon.
\end{eqnarray}
As $\epsilon>0$ was arbitrary it follows that $\| \tau^+|_C\| + \| \tau^-|_C\| \leq \| \tau|_C\|$. 
\end{proof}

\begin{proposition}\label{p:JordanPed}
If $\tau \colon \Ped(A) \to \mathbb C$ is a continuous self-adjoint trace, then there are unique positive traces $\tau^+ ,\tau^- \colon \Ped(A) \to \mathbb C$ such that $\tau = \tau^+ - \tau^-$ and $(\tau^+|_D, \tau^-|_D)$ is the Jordan decomposition of $\tau|_D$ for every hereditary $\cs$-subalgebra $D\subseteq \Ped(A)$. 
\end{proposition}
\begin{proof}
This follows from Lemma \ref{l:Jordan2} and Remark \ref{r:directed}.
\end{proof}

\begin{definition}
    The \emph{Jordan decomposition} $(\tau^+, \tau^-)$ of a continuous self-adjoint trace $\tau \colon \Ped(A) \to \mathbb C$ is the decomposition given in Proposition \ref{p:JordanPed}.
\end{definition}

\begin{corollary}\label{c:contlincomb}
A tracial linear functional $\tau\colon \Ped(A) \to \mathbb C$ is continuous if and only if it is a linear combination of positive traces.
\end{corollary}

\begin{corollary}[Pedersen]\label{c:lattice}
For any $\cs$-algebra $A$, $T_\mathbb{R}(A)$ with positive cone $T_+(A)$ is a real vector lattice.
\end{corollary}
\begin{proof}
    $(T_{\mathbb R}(A), T_+(A))$ is clearly an ordered vector space. For $\tau\in T_{\mathbb R}(A)$, $\tau^+$ is the supremum of $\tau$ and $0$ when restricted to every hereditary $\cs$-subalgebra (Proposition \ref{p:Jordan1}(c)) and thus $\tau^+ = \tau \vee 0$.    
    It is well-known (and straightforward to verify) that this implies that $T_{\mathbb R}(A)$ is a vector lattice with join $\tau_1 \vee \tau_2 = (\tau_1 - \tau_2)^+ + \tau_2$ and meet $\tau_1 \wedge \tau_2 = \tau_1 - (\tau_1 - \tau_2)^+$. 
\end{proof}

The following follows from Lemma \ref{l:Jordan2}.

\begin{corollary}\label{c:herlattice}
    Let $A\subseteq B$ be a hereditary $\cs$-subalgebra. Then the restriction map $T_{\mathbb R}(B) \to T_{\mathbb R}(A)$ is a lattice homomorphism.
\end{corollary}

\section{The dual trace space}

When  $A$ is a unital $\cs$-algebra there is a duality due to Kadison \cite{Kadison-duality} (see also \cite[Chapter 7]{Alfsen-book}) of the Choquet simplex $T_1(A)$ of tracial states and the Archimedian order unit space $\Aff T_1(A)$ of real affine continuous functions on $T_1(A)$. The space $T_{\mathbb R}(A)$ is a real topological vector space, so the right dual space consists of continuous linear functionals. We address how this relates to the space of affine functions on positive traces in Section \ref{s:Aff}.

\begin{definition}
Let $A$ be a $\cs$-algebra. We let $T_{\mathbb R}(A)^\ast$ denote the continuous dual space of $T_{\mathbb R}(A)$, i.e.~the space of weak$^\ast$-continuous functionals $T_{\mathbb R}(A) \to \mathbb R$. We equip $T_{\mathbb R}(A)^\ast$ with the compact-open topology, i.e.~a net $(f_i)_i$ in $T_{\mathbb R}(A)^\ast$ converges to $f$ if and only if it converges uniformly on weak$^\ast$-compact subsets of $T_{\mathbb R}(A)$. We also equip $T_{\mathbb R}(A)^\ast$ with a positive cone  $T_{\mathbb R}(A)^\ast_+$ consisting of the functionals $f$ such that $f(\tau) \geq 0$ for all $\tau\in T_{+}(A)$.

In this way, $T_{\mathbb R}(A)^\ast$ is an ordered topological vector space.
\end{definition}

For every compact subset $K \subseteq T_{\mathbb R}(A)$ we let $\| \cdot\|_K$ denote the seminorm on $T_{\mathbb R}(A)^\ast$ given by
\[
\| f \|_K = \sup_{\tau \in K} | f(\tau)| , \qquad \textrm{for } f\in T_{\mathbb R}(A)^\ast.
\]
Clearly the compact-open topology on $T_{\mathbb R}(A)^\ast$ is the locally convex topology induced by the family of seminorms $\| \cdot \|_K$ where $K$ ranges over the compact subsets of $T_{\mathbb R}(A)$.
In order to understand the  compact-open topology on $T_{\mathbb R}(A)^\ast$, we must first understand (relatively) compact subsets of $T_{\mathbb R}(A)$ as follows.

\begin{lemma}\label{l:compact}
For a subset $K\subseteq T_{\mathbb R}(A)$ the following are equivalent:
\begin{itemize}
\item[(i)] $K$ is relatively compact (i.e.~$K$ has compact closure);
\item[(ii)] every $f\in T_{\mathbb R}(A)^\ast$ is bounded on $K$;
\item[(iii)] $\sup_{\tau\in K} | \tau(a)| < \infty$ for all $a\in \Ped(A)$;
\item[(iv)] $\sup_{\tau\in K} \| \tau|_{D}\| < \infty$ for every $\cs$-subalgebra $D \subseteq \Ped(A)$.
\end{itemize}
\end{lemma}
\begin{proof}
We let $\overline{K}$ denote the closure of $K$.

(i) $\Rightarrow$ (ii): $f(\overline K)\subseteq \mathbb R$ is compact, hence bounded.

(ii) $\Rightarrow$ (iii): It suffices to prove (iii) for self-adjoint $a\in \Ped(A)$, and this case reduces to (ii) since $\tau \mapsto \tau(a)$ defines an element of $T_{\mathbb R}(A)^\ast$.

(iii) $\Rightarrow$ (iv): This follows from the uniform boundedness principle.

(iv) $\Rightarrow$ (i): Let $(\tau_i)_i$ be an ultranet in $\overline K$. Note that (iii) clearly follows from (iv) and hence 
\[
\sup_{\tau\in \overline K} |\tau(a)| = \sup_{\tau \in K} |\tau(a) | <\infty
\]
for all $a\in \Ped(A)$. Thus we may define a linear functional $\tau_0 \colon \Ped(A) \to \mathbb C$ by $\tau_0(a) = \lim_i \tau_i(a)$. Let $D \subseteq \Ped(A)$ be a $\cs$-subalgebra. By (iv)  $\tau_0|_{D}$ is bounded, and thus $\tau_0$ is continuous. Hence $\tau_0 \in T_{\mathbb R}(A)$ and by construction $\tau_i \to \tau_0$, so $\overline K$ is compact.
\end{proof}

Note that there is a canonical linear map $\Ped(A)_{\mathrm{sa}} \to T_{\mathbb R}(A)^\ast$ given by $a \mapsto \hat a$ where $\hat a(\tau) = \tau(a)$ for $\tau\in T_{\mathbb R}(A)$. Clearly $\hat a \in T_{\mathbb R}(A)^\ast_+$ whenever $a\in \Ped(A)_+$. We emphasise for the following theorem that $\Ped(A)$ is considered with the inductive limit topology (Definition \ref{d:locconvex}) and $T_{\mathbb R}(A)^\ast$ is equipped with the compact-open topology.

\begin{theorem}\label{t:PedTcont}
The map $\Ped(A)_{\mathrm{sa}} \to T_{\mathbb R}(A)^\ast$ is a quotient map of locally convex spaces with respect to the topology on $\Ped(A)$ given in Definition \ref{d:locconvex} and the compact-open topology on $T_{\mathbb R}(A)^\ast$.
\end{theorem}
\begin{proof}
Surjectivity is elementary, see for instance \cite[Proposition 2.4.4]{Pedersen-book-analysisnow}.

Next we show continuity. By definition of the locally convex direct limit topology on $\Ped(A)_{\mathrm{sa}}$, this amounts to showing that the restriction $D_{\mathrm{sa}} \to T_{\mathbb R}(A)^\ast$ is continuous (in the norm topology on $D_{\mathrm{sa}}$) for every $\cs$-subalgebra $D\subseteq \Ped(A)$. Let $(a_i)_i$ be a net in $D_{\mathrm{sa}}$ converging to $a$, and let $S\subseteq T_{\mathbb R}(A)$ be a compact subset. By Lemma \ref{l:compact} $M:= \sup_{\tau\in S} \| \tau|_D\| < \infty$, and thus
\begin{equation}
\sup_{\tau\in S} | \hat a_i(\tau) - \hat a(\tau)| = \sup_{\tau\in S} | \tau(a_i - a)| \leq \| a_i - a\| M \to 0. 
\end{equation}
To demonstrate that $T_{\mathbb R}(A)^\ast$ carries the quotient topology, we show that a seminorm $\rho \colon T_{\mathbb R}(A)^\ast \to [0,\infty)$ for which the pre-composition $\tilde \rho \colon \Ped(A)_{\mathrm{sa}} \to [0,\infty)$ is continuous must itself be continuous.  
By the Hahn--Banach theorem, for any $a \in \Ped(A)_{\mathrm{sa}}$ we have 
\[ \tilde \rho(a) = \sup \{ |\tau(a)| : \tau \in (\Ped(A)_{\mathrm{sa}})^\ast, |\tau| \leq \tilde \rho \}. \]
For $\tau \in (\Ped(A)_{\mathrm{sa}})^\ast$, the condition $|\tau| \leq \tilde \rho$ forces $\tau$ to be tracial and therefore says precisely that $\tau$ is in the compact set $K \subseteq T_{\mathbb R}(A)$ given by
\[
K = \{ \tau \in T_{\mathbb R}(A) : |\tau(a)| \leq \rho(\hat a) \text{ for each } a \in \Ped(A)_{\mathrm{sa}} \}.
\]
Thus for each $a \in \Ped(A)_{\mathrm{sa}}$ we have
\[ \rho(\hat a) = \tilde \rho (a) = \sup_{\tau \in K} |\tau(a)| = \| \hat a \|_K, \]
and so $\rho$ is continuous with respect to the compact open topology.
\end{proof}

\begin{corollary}\label{c:seminorm}
    Let $A$ be a $\cs$-algebra and let $\rho \colon \Ped(A)\to [0,\infty)$ be a seminorm such that the restriction $\rho|_D$ is bounded for every $\cs$-subalgebra $D\subseteq \Ped(A)$. 
    Let
    \[
    K_\rho := \{ \tau \in T_{\mathbb R}(A) : |\tau(a)| \leq \rho(b) \textrm{ for all } a,b\in \Ped(A)_{\mathrm{sa}} \textrm{ such that }\hat a = \hat b\}.
    \]
    Then $K_\rho$ is compact, and for every $f\in T_{\mathbb R}(A)^\ast$ and every $\epsilon>0$ there exists an $a\in \Ped(A)_{\mathrm{sa}}$ such that $\hat a = f$ and $\rho(a) \leq \| f\|_{K_{\rho}} +\epsilon$. 
\end{corollary}
\begin{proof}
    Note that $\rho$ is continuous in the locally convex inductive limit topology (Definition \ref{d:locconvex}). Therefore, by Theorem \ref{t:PedTcont}, the  induced quotient seminorm $\hat \rho \colon T_{\mathbb R}(A)^\ast \to [0,\infty)$ given by
    \[
    \hat \rho(f) = \inf\{  \rho(a) : a\in \Ped(A)_{\mathrm{sa}} \textrm{ such that }\hat a = f \}
    \]
    is continuous, and the set $K_\rho$ is the set named $K$ (with respect to $\hat \rho$) in the proof of Theorem \ref{t:PedTcont}. Hence, as seen in this proof, $\| f\|_{K_\rho} = \hat \rho(f)$ so the result follows.
\end{proof}

In the following, $T_{\leq 1}(A) \subseteq T_{\mathbb R}(A)$ denotes the set of contractive positive traces. We note that $\tau(a) = f(\tau)$ for \emph{all} traces $\tau\in T_{\mathbb R}(A)$ below, not just bounded traces. 

\begin{corollary}\label{c:norm}
    Let $A$ be a $\cs$-algebra. For any $f\in T_{\mathbb R}(A)^\ast$ and any $\epsilon >0$ there is an $a\in \Ped(A)_{\mathrm{sa}}$ such that $f = \hat a$ and $\| a\| \leq \sup_{\tau\in T_{\leq 1}(A)} |f(\tau)| + \epsilon$. 
\end{corollary}
\begin{proof}
    Applying Corollary \ref{c:seminorm} with the usual norm $\| \cdot \|$ on $\Ped(A)$,  it is immediate to check that the set $K_{\| \cdot\|}$  is the set of traces of norm at most $1$. As $\| \cdot \|_{K_{\|\cdot\|}} = \| \cdot \|_{T_{\leq 1}(A)}$ the result follows.
\end{proof}

\begin{remark}
    Suppose $A$ is a $\cs$-algebra whose non-zero traces are all unbounded.  From Corollary \ref{c:norm} it follows that every $f\in T_{\mathbb R}(A)^\ast$ can be realised as $f = \hat a$ for a self-adjoint $a\in \Ped(A)_{\mathrm{sa}}$ with arbitrarily small norm. This is easy to see if $A$ is stable, but was somewhat surprising to the authors in general.
\end{remark}

We have a duality theory for $T_{\mathbb R}(A)$ and its dual $T_{\mathbb R}(A)^\ast$. 
The following result shows how one reconstructs $T_{\mathbb R}(A)$ (including the topology and the positive cone) from $T_{\mathbb R}(A)^\ast$. Note that this would be immediate if we used the weak$^\ast$-topology on $T_{\mathbb R}(A)^\ast$ through the isomorphism $T_\mathbb R(A) \cong (T_{\mathbb R}(A)^\ast, w^\ast)^\ast$. Hence, for the sake of clarity in the statement of the following proposition, we emphasise the compact--open topology $\mathrm{c.o.}$ on $T_{\mathbb R}(A)^*$. We consider $(T_{\mathbb R}(A)^\ast, {\mathrm{c.o.}})^\ast$ as an ordered topological vector space when equipped with the weak$^\ast$-topology and the positive cone of functionals $g\in (T_{\mathbb R}(A)^\ast, {\mathrm{c.o.}})^\ast$ for which $g(T_{\mathbb R}(A)^\ast_+) \subseteq [0,\infty)$.

\begin{proposition}\label{p:dualiso}
Let $A$ be a $\cs$-algebra. 
The canonical map $T_{\mathbb R}(A) \to (T_{\mathbb R}(A)^\ast, \mathrm{c.o.})^\ast$
is an isomorphism of ordered topological vector spaces.
\end{proposition}
\begin{proof}
Let $\omega \colon T_{\mathbb R}(A) \to (T_{\mathbb R}(A)^{\ast}, {\mathrm{c.o.}})^\ast$ denote the canonical map (which is clearly well-defined since any net in $T_{\mathbb R}(A)^{\ast}$ which converges in the compact-open topology also converges pointwise). The functionals in $T_{\mathbb R}(A)^\ast$ separate the points of $T_{\mathbb R}(A)$ by Theorem \ref{t:PedTcont}, so $\omega$ is injective. For surjectivity, when $h \colon T_{\mathbb R}(A)^\ast \to \mathbb R$ is a continuous functional, so is its composition with the map $\Ped(A)_{\mathrm{sa}} \to T_{\mathbb R}(A)^\ast$ by Theorem \ref{t:PedTcont}. This composition extends by linearity to a self-adjoint continuous trace $\tau$ on $\Ped(A)$ satisfying $\omega(\tau) = h$, so $\omega$ is a bijection.

That $\omega$ is a homeomorphism follows straightforwardly from Theorem \ref{t:PedTcont}. 

Finally if $\tau\in T_+(A)$ then $\omega(\tau)(f) = f(\tau) \geq 0$ for every $f\in T_{\mathbb R}(A)^\ast_+$ (by definition of $T_{\mathbb R}(A)^\ast_+$). Hence $\omega(\tau)$ is positive. Conversely, if $\tau\in T_{\mathbb R}(A)$ such that $\omega(\tau)$ is positive, and if $a\in \Ped(A)_+$, then $\hat a \in T_{\mathbb R}(A)^\ast_+$ and therefore $\tau(a) = \omega(\tau)(\hat a) \geq 0$, so $\tau \in T_+(A)$.
\end{proof}

\begin{theorem}\label{t:duality}
Let $A$ and $B$ be $\cs$-algebras. Taking duals induces a one-to-one correspondence between continuous linear (positive) maps $T_{\mathbb R}(B) \to T_{\mathbb R}(A)$ and continuous linear (positive) maps $T_{\mathbb R}(A)^\ast \to T_{\mathbb R}(B)^\ast$. 
\end{theorem}
\begin{proof}
    We will use Theorem \ref{t:PedTcont} several times to identify elements of $T_{\mathbb R}(D)^\ast$ with self-adjoint elements in $\Ped(D)$. 
    
    Let $\gamma \colon T_{\mathbb R}(B) \to T_{\mathbb R}(A)$ be a continuous linear (positive) map. We claim that $\gamma^\ast \colon T_{\mathbb R}(A)^\ast \to T_{\mathbb R}(B)^\ast$ given by
    \begin{equation}
        \gamma^\ast(\hat a) (\tau_B) = \gamma(\tau_B)(a)
    \end{equation}
    for $a\in \Ped(A)_{\mathrm{sa}}$ and $\tau_B\in T_{\mathbb R}(B)$ is well-defined, continuous (and positive). 

    To see that it is well-defined, it suffices to show for $a\in \Ped(A)_{\mathrm{sa}}$ that $T_{\mathbb R}(B) \to \mathbb R : \tau_B \mapsto \gamma(\tau_B)(a)$ is weak$^\ast$-continuous. But this is clear since $\gamma$ is weak$^\ast$-weak$^\ast$-continuous.  
    
    Clearly $\gamma^\ast$ is linear. To see that it is continuous (with respect to the compact-open topology on both $T_{\mathbb R}(A)^\ast$ and $T_{\mathbb R}(B)^\ast$), let $(a_i)_i$ be a net in $\Ped(A)_{\mathrm{sa}}$ such that $(\hat a_i)_i$ converges uniformly on compact sets to $\hat a$. Let $S\subseteq T_{\mathbb R}(B)$ be compact. Then $\gamma(S)$ is compact by continuity of $\gamma$ and thus 
    \begin{equation}
        \sup_{\tau \in S} | \gamma^\ast(\hat a_i - \hat a)(\tau_B)| = \sup_{\tau_A \in \gamma(S)} | (\hat a_i - \hat a)(\tau_A)| \to 0,
    \end{equation}
    so $\gamma^\ast$ is continuous. Finally, if $\gamma$ is positive and if $f\in T_{\mathbb R}(A)^\ast_+$ and $\tau_B\in T_+(B)$ then $\gamma^\ast(f)(\tau_B) = f(\gamma(\tau_B)) \geq 0$, so $\gamma^\ast$ is positive.

    Now, suppose $\eta \colon T_{\mathbb R}(A)^\ast \to T_{\mathbb R}(B)^\ast$ is a continuous linear (and positive) map. We define $\eta^\ast \colon T_{\mathbb R}(B) \to T_{\mathbb R}(A)$ by $\eta^\ast(\tau_B)(a)= \eta(\hat a) (\tau_B)$ (this is the dual map when applying Proposition \ref{p:dualiso}). We again check that this is well-defined, continuous (and positive). 

    To see that it is well-defined, we clearly have that $a \mapsto \eta(\hat a)(\tau_B)$ is a trace on $\Ped(A)_{\mathrm{sa}}$ so it extends canonically to a self-adjoint trace on $\Ped(A)$. To see that this trace is continuous, let $(a_i)_i$ be a net in $\Ped(A)_{\mathrm{sa}}$ converging to $a\in \Ped(A)_{\mathrm{sa}}$. By Theorem \ref{t:PedTcont}, $\hat a_i \to \hat a$ in $T_{\mathbb R}(A)^\ast$ and thus $\eta(\hat a_i) \to \eta(\hat a)$. It follows that 
    \begin{equation}
        \eta(\hat a_i) (\tau_B) \to \eta(\hat a) (\tau_B)
    \end{equation} 
    for all $\tau_B$ and thus $\eta^\ast(\tau_B) \in T_{\mathbb R}(A)$ for all $\tau_B \in T_{\mathbb R}(B)$. 

    For continuity of $\eta^\ast$ let $\tau_i \to \tau$ in $T_{\mathbb R}(B)$. For $a\in \Ped(A)_{\mathrm{sa}}$ we pick $b\in \Ped(B)_{\mathrm{sa}}$ such that $\hat b = \eta(\hat a)$. Then
    \begin{equation}
        \eta^{\ast}(\tau_i) (a) = \tau_i(b) \to \tau(b) = \eta^\ast(\tau) (a) 
    \end{equation}
    so $\eta^\ast$ is continuous. 

    If $\eta$ is positive, and if $\tau_B \in T_+(B)$ and $a\in A_+$, then $\eta(\hat a) \geq 0$ and thus
    \begin{equation}
        \eta^\ast(\tau_B)(a) = \eta(\hat a) (\tau_B) \geq 0,
    \end{equation}
    so $\eta^\ast$ is positive.

    It is elementary to check that $\gamma^{\ast \ast} = \gamma$ for every continuous linear map $\gamma \colon T_{\mathbb R}(B) \to T_{\mathbb R}(A)$ and that $\eta^{\ast \ast} = \eta$ for every continuous linear map $\eta \colon T_{\mathbb R}(A)^\ast \to T_{\mathbb R}(B)^\ast$ which finishes the proof.
\end{proof}

It is easy to see that any $\ast$-homomorphism $\phi \colon A \to B$ induces a continuous linear positive map $T_{\mathbb R}(B) \to T_{\mathbb R}(A)$, and that $T_{\mathbb R}$ is a contravariant functor. By duality $T_{\mathbb R}(\cdot)^\ast$ is a covariant functor, and any $\ast$-homomorphism induces a continuous linear positive map.

\begin{remark}\label{r:aue}
    Let $\phi, \psi \colon A \to B$ be $\ast$-homomorphisms which are approximately Murray--von Neumann equivalent in the sense of \cite{Gabe-O2class}, i.e.~there is a net $(v_i)_i$ of contractive multipliers of $B$ such that $\lim_i v_i \phi(a) v_i^\ast = \psi(a)$ and $\lim_i v_i^\ast \psi(a) v_i = \phi(a)$ for all $a\in A$ (this is implied by approximate unitary equivalence). Then $T_{\mathbb R}(\phi) = T_{\mathbb R}(\psi)$, and so also $T_{\mathbb R}(\phi)^\ast = T_{\mathbb R}(\psi)^\ast$ by duality.  
    
    In fact, as any $\tau\in T_+(B)$ is norm lower semicontinuous by \cite[Proposition 5.6.7]{Pedersen-book-automorphism}, we have for any positive $a\in \Ped(A)$ that
    \[
    \tau(\psi(a)) \leq \liminf_i \tau(v_i \phi(a) v_i^\ast) = \liminf_i \tau(\phi(a)^{1/2} v_i^\ast v_i \phi(a)^{1/2}) \leq \tau(\phi(a))
    \]
    and similarly $\tau(\phi(a)) \leq \tau(\psi(a))$. As $T_{\mathbb R}(B)$ and $\Ped(A)$ are spanned by $T_+(B)$ and $\Ped(A)_+$ respectively, the equality $T_{\mathbb R}(\phi) = T_{\mathbb R}(\psi)$ follows.
\end{remark}

\section{Extending traces from subalgebras}

In this section we obtain sufficient conditions for when $T_{\mathbb R}(A)$ is completely captured by a $\ast$-subalgebra $D \subseteq \Ped(A)$. 

We will use the following lemma, which is a slight strengthening of the technique used in \cite[Proposition 5.6.2]{Pedersen-book-automorphism} to show that $\Ped(A)$ contains the hereditary $\cs$-subalgebra generated by any finite set in $\Ped(A)$.

\begin{lemma}\label{l:trick}
Let $A$ be a $\cs$-algebra. Given $a_1, \dots, a_m\in A$ we set $a := \left(\sum_{j=1}^m a_j^\ast a_j\right)^{1/2}$. For each $j=1,\dots,m$ there is a unique contractive right $A$-linear map $\Phi_j \colon \overline{aA} \to \overline{a_j A}$ such that $\Phi_j(ab) = a_j b$ for all $b\in A$, and these satisfy $x^\ast y = \sum_{j=1}^m \Phi_j(x)^\ast \Phi_j(y)$ for all $x,y\in \overline{aA}$. 
\end{lemma}
\begin{proof}
Since $a_j^\ast a_j \leq a^2$ we have for $b\in A$ that
\[
\| a_j b \| = \| b^\ast a_j^\ast a_j b\|^{1/2} \leq \| b^\ast a^2 b\|^{1/2} \leq \| a b\|.
\] 
Hence the right $A$-linear map $aA \to \overline{a_j A}$ given by $ab \mapsto a_j b$ for $b\in A$ extends uniquely to a contractive right $A$-linear map $\Phi_j \colon \overline{aA} \to \overline{a_jA}$. 

If $x,y\in \overline{aA}$ let $x_n,y_n \in A$ such that $ax_n \to x$ and $ay_n \to y$. Then $x_n^\ast a^2 y_n = \sum_{j=1}^m x_n^\ast a_j^\ast a_j y_n = \sum_j \Phi_j(ax_n)^\ast \Phi_j(ay_n)$ and thus
\[
x^\ast y = \lim_{n\to \infty} x_n^\ast a^2 y_n = \lim_{n\to \infty} \sum_{j=1}^m \Phi_j(ax_n)^\ast \Phi_j(ay_n) = \sum_{j=1}^m \Phi_j(x)^\ast \Phi_j(y). \qedhere
\]
\end{proof}

We record the following well-known lemma for later use. As we cannot find a reference for either part, we provide proofs.

\begin{lemma}\label{l:Xfull}
    Let $A$ be a $\cs$-algebra and $X \subseteq A$ be a subset that generates $A$ as a two-sided closed ideal. 
    \begin{itemize}
        \item [(a)] For every $a\in A_+$ and $\epsilon >0$ there exist $m\in \mathbb N$, $x_1,\dots, x_m \in X$ and $c_1,\dots, c_m \in A$ such that $\| a - \sum_{j=1}^m c_j^\ast x_j^\ast x_j c_j \| <\epsilon$.
        \item[(b)] For every $a\in \Ped(A)_+$ there exist $m\in \mathbb N$, $x_1,\dots, x_m \in X$ and $c_1,\dots, c_m \in A$ such that $a = \sum_{j=1}^m c_j^\ast x_j^\ast x_j c_j$.
    \end{itemize}
\end{lemma}
\begin{proof}
    (a): Since 
    \begin{equation}
    (u+v)^\ast (u+v) \leq (u+v)^\ast (u+v) + (u-v)^\ast (u-v) = 2u^\ast u + 2 v^\ast v,
    \end{equation}
    we have $(\sum_{j=1}^m u_j)^\ast ( \sum_{j=1}^m u_j) \leq \sum_j 2^j u_j^\ast u_j$ by induction. Now, as $AXA$ spans a dense ideal in $A$,  $a^{1/2}$ can be approximated by $\sum_{j=1}^m b_j x_j d_j$ where $x_j\in X$ and $b_j,d_j\in A$. Hence $a$ is approximately 
    \begin{equation}
        a_0 := (\sum_j b_j x_j d_j)^\ast (\sum_j b_j x_j d_j) \leq a_1 := \sum_j 2^j \| b_j\|^2 d_j^\ast x_j^\ast x_j d_j.
    \end{equation} Write $a_0^{1/2} = a_1^{1/4} u$ for some $u\in A$ by \cite[Proposition 1.4.5]{Pedersen-book-automorphism}. For large $n$, $a_0$ is approximately 
\begin{equation}
u^\ast a_1^{1/4} (a_1 + 1/n)^{-1/2} a_1(a_1+1/n)^{-1/2} a_1^{1/4} u = \sum_j c_j^\ast x_j^\ast x_j c_j
\end{equation}
where $c_j = 2^{j/2} \| b_j\| d_j (a_1+ 1/n)^{-1/2} a_1^{1/4} u$.

(b): Let $a_1,\dots,a_n \in \Ped(A)_0$ such that $a \leq \sum_j a_j$. By \cite[Proposition 1.4.10]{Pedersen-book-automorphism} there are $v_1,\dots, v_n \in A$ such that $v_jv_j^\ast \leq a_j$ (and in particular, $v_jv_j^\ast \in \Ped(A)_0$), and $a = \sum_{j=1}^n v_j^\ast v_j$. Let $e_j \in \Ped(A)_+$ so that $e_j v_j = v_j$. By part (a) we may find $x_{j,k} \in X$ (finitely many) and $z_{j,k} \in A$ such that $\| e_j - \sum_k z_{j,k}^\ast x_{j,k}^\ast x_{j,k} z_{j,k} \| < 1/2$. By \cite[Lemma 2.2]{KirchbergRordam-absorbingOinfty} find $d_{j} \in A$ so that $\sum_k d_{j}^\ast z_{j,k}^\ast x_{j,k}^\ast x_{j,k} z_{j,k} d_{j} = e_j' := (e_j - \tfrac{1}{2})_+$. Then $e_j' v_j = \tfrac{1}{2}v_j$.\footnote{In fact, if $f\in C_0((0, \|e_j\|\,])$ then $f(e_j) v_j = f(1) v_j$, since this can be verified when $f$ is a polynomial with no constant term.} Letting $c_{j,k} := \sqrt{2} z_{j,k} d_{j} v_j$ we get 
\[
\sum_{j,k} c_{j,k}^\ast x_{j,k}^\ast x_{j,k} c_{j,k} = 2\sum_j v_j^\ast e_j' v_j = a. \qedhere
\]
\end{proof}

The following proposition gives a sufficient condition under which traces on $A$ can be studied through a $\ast$-subalgebra $A_0 \subseteq \Ped(A)$. We have stated the result somewhat generally (and technically) so that it can hopefully be useful in multiple contexts, such as when $A_0$ is dense or a full hereditary $\cs$-subalgebra. 

In the following, we consider a $\ast$-subalgebra $A_0 \subseteq \Ped(A)$ with a subset $X\subseteq A_0$ such that $XA_0 = A_0$, $X$ generates $A$ as a two-sided closed ideal, and $xA_0y^\ast$ is dense in $\overline{xAy^\ast}$ for all $x,y\in X$. Under these assumptions, we define $T_{\mathbb R}(A_0,X)$ as the set
\[
\{ \tau_0 \colon A_0 \to \mathbb C : \tau_0 \textrm{ self-adjoint, tracial with $\tau_0|_{xA_0x^\ast}$ bounded for all $x\in X$}\}. 
\]
Then $T_{\mathbb R}(A_0,X)$ is an ordered (real) vector space with cone of traces $\tau_0$ that are positive in the sense that $\tau_0(a^\ast a) \geq 0$ for all $a\in A_0$. 

We equip $T_{\mathbb R}(A_0,X)$ with the weak topology coming from evaluation at elements in $\overline{xAx^\ast}$ for $x\in X$ (which makes sense since each $\tau_0$ extends uniquely to a bounded trace on $\overline{xA_0 x^\ast} = \overline{xAx^\ast}$). Alternatively, a net $(\tau_i)_i$ in $T_{\mathbb R}(A_0,X)$ converges to $\tau_0\in T_{\mathbb R}(A_0,X)$ if for every $x\in X$ and sequence $(a_n)_n$ in $A_0$ such that $(xa_n x^\ast)_n$ is a $\|\cdot\|$-Cauchy sequence, we have
\[
\lim_i \lim_{n\to \infty} \tau_i(xa_nx^\ast) = \lim_{n\to \infty} \tau_0(xa_n x^\ast).
\]

\begin{proposition}\label{p:extendtraces}
Let $A$ be a $\cs$-algebra, let $A_0 \subseteq \Ped(A)$ be a $\ast$-sub\-algebra with a subset $X\subseteq A_0$ such that $X A_0 = A_0$, $X$ generates $A$ as a two-sided closed ideal, and $xA_0y^\ast$ is dense in $\overline{xAy^\ast}$ for all $x,y\in X$. Then the restriction map
\[
T_{\mathbb R}(A) \to T_{\mathbb R}(A_0,X) , \quad \tau \mapsto \tau|_{A_0}
\]
is an isomorphism of ordered topological vector spaces. 
\end{proposition}
\begin{proof}
The map is clearly well-defined since any $\tau\in T_{\mathbb R}(A)$ is bounded on $\overline{xAx^\ast} \subseteq \Ped(A)$ (Lemma \ref{l:Pedher}) for all $x\in X$ by the definition of the topology on $\Ped(A)$. It is also clearly continuous. We show that the map is bijective.

For injectivity, suppose $\tau_1, \tau_2 \in T_{\mathbb R}(A)$ are such that $\tau_1|_{A_0} = \tau_2|_{A_0}$. Let $a\in \Ped(A)_+$. It suffices to show that $\tau_1(a) = \tau_2(a)$. 

By Lemma \ref{l:Xfull} we pick $x_1,\dots,x_m \in X$ and $c_j \in \overline{x_j A}$ such that $a = \sum_j c_j^\ast c_j$. As $c_j c_j^\ast \in \overline{x_j A x_j^\ast} = \overline{x_j A_0 x_j^\ast}$, we pick $y_{j,n}$ in $A_0$ such that $x_j y_{j,n} x_j^\ast \to c_j c_j^\ast$ as $n\to \infty$. As $\tau_1$ and $\tau_2$ are bounded on $\overline{x_j A x_j^\ast}$ we get
\[
\tau_1(a) = \sum_j \tau_1(c_j c_j^\ast) = \lim_{n\to \infty} \tau_1(x_j y_{j,n} x_j^\ast) = \lim_{n\to \infty} \tau_2(x_j y_{j,n} x_j^\ast) =  \tau_2(a).
\]

For surjectivity of the map, suppose $\tau_0 \colon A_0 \to \mathbb C$ is a self-adjoint trace such that $\tau_0|_{xA_0x^\ast}$ is bounded for all $x\in X$. Hence $\tau_0$ extends canonically to $\overline{xAx^\ast}$ for each $x\in X$ and we denote this extension by $\tau_0^x$. Note that if $x,y\in X$ and $r \in \overline{xAy^\ast}$ we have $\tau_0^x(rr^\ast) = \tau_0^y(r^\ast r)$. In fact, since $xA_0 y^\ast$ is dense in $\overline{xAy^\ast}$, we may find $z_n \in A_0$ such that $xz_ny^\ast \to r$, and thus
\begin{equation}\label{eq:tau0xrr*}
\tau_0^x(rr^\ast) = \lim_{n\to \infty} \tau_0(xz_ny^\ast y z_n^\ast x^\ast) = \lim_{n\to \infty} \tau_0(yz_n^\ast x^\ast x z_n y^\ast) = \tau_0^y(r^\ast r).
\end{equation}
We now construct an extension $\tau$ of $\tau_0$ to $\Ped(A)$. Let $a\in \Ped(A)_+$ and find $x_1,\dots,x_m \in X$ and $c_j\in \overline{x_j A}$ such that $a = \sum_j c_j^\ast c_j$ by Lemma \ref{l:Xfull}. We define $\tau(a) = \sum_j \tau_0^{x_j}(c_j c_j^\ast)$. 

To see that this is well-defined, let $y_1,\dots,y_n \in X$ and $d_j\in \overline{y_j A}$ such that $a= \sum_{k=1}^n d_k^\ast d_k$. Since $a = \sum_j c_j^\ast  c_j = \sum_k d_k^\ast d_k$ there are by \cite[Lemma 5.2.5]{Pedersen-book-automorphism} $r_{j,k} \in A$ such that
\[
c_j c_j^\ast  = \sum_k r_{j,k} r_{j,k}^\ast, \qquad d_k d_k^\ast   = \sum_j r_{j,k}^\ast r_{j,k}.
\]
Note that $r_{j,k} \in \overline{x_j A y_k^\ast}$. By \eqref{eq:tau0xrr*} we have
\[
\sum_{j=1}^m \tau_0^{x_j}(c_jc_j^\ast ) = \sum_{j,k} \tau_0^{x_j}(r_{j,k} r_{j,k}^\ast) = \sum_{j,k} \tau_0^{y_k}(r_{j,k}^\ast r_{j,k}) = \sum_{k=1}^n \tau_0^{y_k}(d_k d_k^\ast).
\]
Hence $\tau\colon \Ped(A)_+\to \mathbb R$ is well-defined, and extends to a (unique) self-adjoint linear functional $\tau \colon \Ped(A) \to \mathbb C$. 

From the definition of $\tau$ it is immediate that if $a\in \Ped(A)_+$ and $u\in \widetilde A$ is a unitary, then $\tau (u^\ast a u) = \tau(a)$ (by considering the definition with $c_j u$ instead of $c_j$). Hence it follows that $\tau$ is tracial. 

To see that $\tau|_{A_0} = \tau_0$ let $a\in A_0$. As $X\cdot A_0 = A_0$ we write $a= b^\ast c$ for $b,c\in A_0$. Then $a = \tfrac{1}{4}\sum_{j=0}^3 i^j (b+i^j c)^\ast (b+i^j c)$ (the polarisation identity). By considering each term separately, we may assume that $a= b^\ast b$ with $b\in A_0$. As $X \cdot A_0 = A_0$ we find $x\in X$ and $b_0 \in A_0$ such that $b = xb_0$. Then $a= b_0^\ast x^\ast x b_0$ and thus $\tau(a) = \tau_0(xb_0b_0^\ast x^\ast) = \tau_0(a)$.

Now, to see that $\tau$ is continuous let $B\subseteq \Ped(A)$ be a $\sigma$-unital hereditary $\cs$-subalgebra. Let $h\in B$ be strictly positive and find $x_1,\dots,x_m\in X$ and $c_j\in \overline{x_j A}$ such that $h = \sum_{j=1}^m c_j^\ast c_j$. By Lemma \ref{l:trick} we find contractive right $A$-linear maps $\Phi_j \colon \overline{hA} \to \overline{c_j A} \subseteq \overline{x_j A}$ such that $b^\ast c = \sum_{j=1}^m \Phi_j(b)^\ast \Phi(c)$ for all $b,c\in \overline{hA}$. For $a\in B_+$ with $\| a\| \leq 1$ we have $a= \sum_j \Phi_j(a^{1/2})^\ast \Phi_j(a^{1/2})$ with $\Phi_j(a^{1/2}) \in \overline{x_j A}$ and thus $\tau(a) = \sum_j \tau_0^{x_j}(\Phi_j(a^{1/2}) \Phi_j(a^{1/2})^\ast)$. Since $\Phi_j(a^{1/2})\Phi_j(a^{1/2})^\ast \in \overline{x_j A x_j^\ast}$ is contractive for each $j$ and $\tau_0^{x_j}$ is bounded, we get
\[
|\tau(a)| \leq \sum_{j=1}^m |\tau_0^{x_j}(\Phi_j(a^{1/2}) \Phi_j(a^{1/2})^\ast)| \leq \sum_{j=1}^m \| \tau_0^{x_j}\|.
\] 
Hence $\tau|_B$ is bounded, and therefore $\tau$ is continuous. We have therefore shown that $T_{\mathbb R}(A) \to T_{\mathbb R}(A_0,X)$ is a continuous positive bijective linear map. It remains to show that the inverse is continuous and positive.

If $\tau_0\in T_{\mathbb R}(A_0,X)$ is positive, then it is immediate that the extensions $\tau_0^x$ are positive for all $x\in X$, so it is easily seen that the trace $\tau$ constructed in the surjectivity part of the proof is positive.

Finally, suppose $(\tau_i)_i$ is a net in $T_{\mathbb R}(A)$ such that $\tau_i|_{A_0}$ converges to $\tau|_{A_0}$. Let $a\in \Ped(A)_+$ and use Lemma \ref{l:Xfull}(b) to find $x_1,\dots, x_m \in X$ and $c_1,\dots, c_m \in A$ so that $a= \sum_{j=1}^m c_j^\ast x_j^\ast x_j c_j$. Then 
\[
\tau_i(a) = \sum_j \tau_i|_{A_0}^{x_j}(x_j c_j c_j^\ast x_j^\ast) \to \sum_j \tau|_{A_0}^{x_j}(x_j c_j c_j^\ast x_j^\ast) = \tau(a)
\]
so $\tau_i \to \tau$.
\end{proof}

\begin{corollary}\label{c:Tfullher}
Let $A \subseteq B$ be a full hereditary $\cs$-subalgebra. The restriction map $T_{\mathbb R}(B) \to T_{\mathbb R}(A)$ is an isomorphism of ordered topological vector spaces, and so is the induced map $T_{\mathbb R}(A)^\ast \to T_{\mathbb R}(B)^\ast$.
\end{corollary}
\begin{proof}
    This follows from Proposition \ref{p:extendtraces} (use $\Ped(A)$ as both the $\ast$-sub\-algebra and $X$).
    The final part of the corollary follows by duality (Theorem \ref{t:duality}).
\end{proof}

In what follows we let $A\otimes B$ denote the spatial (also known as minimal) tensor product of the $\cs$-algebras $A$ and $B$. We let $\Ped(A) \odot \Ped(B)$ denote the algebraic tensor product of the Pedersen ideals which canonically sits inside $\Ped(A \otimes B)$ as a $\ast$-subalgebra.

\begin{corollary}
Let $A$ and $B$ be $\cs$-algebras with traces $\tau_A \in T_{\mathbb R}(A)$ and $\tau_B \in T_{\mathbb R}(B)$. Then there is a unique trace $\tau_A \otimes \tau_B \in T_{\mathbb R}(A \otimes B)$ that satisfies $(\tau_A \otimes \tau_B)(a \otimes b) = \tau_A(a) \tau_B(b)$ for each $a \in \Ped(A)$ and $b \in \Ped(B)$ and is positive if both $\tau_A$ and $\tau_B$ are positive. Moreover, for any $\tau_B \in T_{\mathbb R}(B)$ the assignment $\tau_A \mapsto \tau_A \otimes \tau_B \colon T_{\mathbb R}(A) \to T_{\mathbb R}(A \otimes B)$ is continuous. If $\tau_B$ is the unique non-zero positive trace in $T_{\mathbb R}(B)$ up to scaling then this is an isomorphism of ordered topological vector spaces.
\end{corollary}
\begin{proof}
We use Proposition \ref{p:extendtraces} with $A_0 = \Ped(A) \odot \Ped(B) \subseteq A \otimes B$ and $X = \{ a \otimes b \mid a \in \Ped(A), b \in \Ped(B) \}$. To check that $X A_0 = A_0$, let $\sum_{i=1}^n a_i \otimes b_i \in A_0$ and apply Cohen factorisation to $\cs(a_1,\dots,a_n) \subseteq \Ped(A)$ and $\cs (b_1,\dots,b_n)$ acting on their $n$th powers to find $x \in X$ and $y \in A_0$ with $xy = \sum_i a_i \otimes b_i$. The other conditions are clear because $A_0 \subseteq \Ped(A \otimes B)$ is a dense $*$-subalgebra. Define a trace $\tau_0 \colon A_0 \to \mathbb C$ by setting
\[ \tau_0 \colon \sum_i a_i \otimes b_i \mapsto \sum_i \tau_A(a_i) \tau_B(b_i). \]
We claim that this is bounded on $x A_0 x^*$ for any $x = a \otimes b \in X$. Let $D_a \subseteq \Ped(A)$ and $D_b \subseteq \Ped(B)$ be the hereditary $\cs$-subalgebras generated by $a$ and $b$. Then $x A_0 x^* \subseteq D_a \otimes D_b$, upon which $\tau_0 = \tau_A |_{D_a} \otimes \tau_B |_{D_b}$ is bounded. We conclude by Proposition \ref{p:extendtraces} that there is a unique trace in $T_{\mathbb R}(A \otimes B)$ extending $\tau_0$.

Given $\tau_B \in T_{\mathbb R}(B)$, we check the continuity of $\tau_A \mapsto \tau_A \otimes \tau_B \colon T_{\mathbb R}(A) \to T_{\mathbb R}(A \otimes B)$ against $\hat x$ for an arbitrary $x \in \Ped(A \otimes B)_{\mathrm{sa}}$. By Lemma \ref{l:Xfull} (b) we may assume without loss of generality that $x = d^*(a \otimes b)^2d$ for some $d \in A \otimes B$, $a \in \Ped(A)_+$ and $b \in \Ped(B)_+$. Then $\hat x = \hat y$ for $y = (a \otimes b)dd^*(a \otimes b)$, which is an element of $\overline{(a \otimes b)(A \otimes B)(a \otimes b)} = \overline{aAa} \otimes \overline{bBb}$. Using the slice map $\id \otimes \tau_B \colon \overline{aAa} \otimes \overline{bBb} \to \overline{aAa}$, we compute $\hat x(\tau_A \otimes \tau_B) = (\tau_A \otimes \tau_B)(y) = \tau_A((\id \otimes \tau_B)(y))$, which is continuous in $\tau_A$. This argument also shows that $\tau_A \otimes \tau_B$ is positive if $\tau_A$ and $\tau_B$ are both positive.

Suppose $\tau_B$ is the unique non-zero positive trace in $T_{\mathbb R}(B)$ up to scaling. We construct a continuous inverse to $\tau_A \mapsto \tau_A \otimes \tau_B \colon T_{\mathbb R}(A) \to T_{\mathbb R}(A \otimes B)$ as follows. Given $\tau \in T_{\mathbb R}(A\otimes B)$ and $a\in \Ped(A)_{\mathrm{sa}}$ the self-adjoint trace $\Ped(B) \to \mathbb C$ given by $b\mapsto \tau(a \otimes b)$ is continuous in the topology from Definition \ref{d:locconvex}. Therefore there is a real scalar $\sigma_\tau(a)$ such that $\tau(a\otimes b) = \sigma_\tau(a) \tau_B(b)$ for all $b\in \Ped(B)$. Letting $b_0 \in \Ped(B)_+$ such that $\tau_B(b_0)=1$ we have $\sigma_\tau(a) = \tau(a \otimes b_0)$ and thus $\sigma_\tau \in T_{\mathbb R}(A)$ which is positive if $\tau$ is positive.
By construction we have $\tau = \sigma_\tau \otimes \tau_B$ and $\sigma_{\tau_A \otimes \tau_B} = \tau_A$.  
As $\sigma_\tau = \tau(-\otimes b_0)$ it follows that $\tau \mapsto \sigma_\tau$ is continuous.
\end{proof}

In Section \ref{s:groupoid} we will also see applications to twisted étale groupoid $\cs$-algebras.

\section{Affine functions on positive traces}\label{s:Aff}

For a convex subset $K$ of a topological vector space we let $\Aff K$ denote the affine continuous real functions on $K$. If $0 \in K$ we let $\Aff_0 K$ denote the subspace of $\Aff K$ of functions vanishing at $0$. These have a positive cone consisting of functions $f$ such that $f(K) \subseteq [0,\infty)$.

We let $T_1(A), T_{\leq 1}(A), T_+^{\mathrm{b}}(A), T_{\mathbb R}^{\mathrm{b}}(A) \subseteq A^\ast$ denote the convex sets of boun\-ded traces which are respectively states, positive contractions, positive, and self-adjoint. We equip each with the weak$^\ast$-topology from $A^\ast$. Note that on unbounded sets of traces this topology (at least priori) differs from the weak$^\ast$-topology coming from $\Ped(A)^\ast$. 

Each of these has a corresponding function space, $\Aff T_1(A), \Aff_0 T_{\leq 1}(A)$, $\Aff_0 T_+^{\mathrm{b}}(A)$ and $T_{\mathbb R}^{\mathrm{b}}(A)^\ast$ respectively. We consider each of these as an ordered topological vector space with the topology induced by the supremum norm with respect to the relevant subset of contractive traces and positive cone of functions which are positive on positive traces. These topologies coincide with the compact-open topology. Through restriction there are natural linear maps
\[ T^{\mathrm{b}}_{\mathbb R}(A)^\ast \to \Aff_0 T^{\mathrm{b}}_+(A) \to \Aff_0 T_{\leq 1}(A) \to \Aff T_1(A) \]
which are continuous, injective and order-preserving. The following is well-known to experts.

\begin{proposition}\label{p:Affbounded}
    Let $A$ be a $\cs$-algebra.  
    Then the canonical maps 
    \begin{equation}\label{eq:Aff}
        T_{\mathbb R}^{\mathrm{b}}(A)^\ast \to \Aff_0 T_+^{\mathrm{b}}(A) \to \Aff_0 T_{\leq 1}(A)
    \end{equation} 
    are isomorphisms of ordered topological vector spaces.
        
    If $T_1(A)$ is compact then the canonical map 
    \begin{equation}
        T_{\mathbb R}^{\mathrm{b}} (A)^\ast \to \Aff T_1(A)
    \end{equation} 
    is an isomorphism of ordered topological vector spaces.
\end{proposition}
\begin{proof}
   The map $\Aff_0 T_+^{\mathrm{b}}(A) \to \Aff_0 T_{\leq 1}(A)$ is clearly injective, and the map $T_{\mathbb R}^{\mathrm{b}}(A)^\ast \to \Aff_0 T_+^{\mathrm{b}}(A)$ is injective by the Jordan decomposition (the Jordan decomposition of bounded traces are again bounded traces (Proposition \ref{p:Jordan1})). To see that the maps in \eqref{eq:Aff} are surjective, it suffices to show that the composition is surjective. Letting $f\in \Aff_0 T_{\leq 1}(A)$, this extends uniquely to a linear functional $\tilde f \colon T_{\mathbb R}^{\mathrm{b}}(A) \to \mathbb R$ by $\tilde f (\lambda_1 \tau_1- \lambda_2 \tau_2) = \lambda_1 f(\tau_1) - \lambda_2 f(\tau_2)$ for $\tau_1,\tau_2\in T_1(A)$ and $\lambda_1,\lambda_2 \in [0,\infty)$. To see that $\tilde f\in T_{\mathbb R}^{\mathrm{b}}(A)^\ast$ we must show that $\tilde f$ is continuous.

    Let $A_0 = \{ a \in A : \tau(a) = 0 \textrm{ for all }\tau \in T_{\mathbb R}^{\mathrm{b}}(A)\}$. Then $A_0\subseteq A$ is a closed $\ast$-invariant subspace and $T_{\mathbb R}^{\mathrm{b}}(A) \cong ((A/A_0)_{\mathrm{sa}})^\ast$ canonically (the dual space of the real Banach space $(A/A_0)_{\mathrm{sa}}$). By the Krein--Smulian theorem it suffices to show that $\tilde f$ is continuous on the closed unit ball of $((A/A_0)_{\mathrm{sa}})^\ast$ (see for instance \cite[Corollary 2.7.9]{Megginson}).

    Let $(\tau_i)_i$ be an ultranet in $T_{\mathbb R}^{\mathrm{b}}(A)$ such that $\| \tau_i\| \leq 1$.   Let $(\tau_i^{+}, \tau_i^-)$ be the Jordan decomposition of $\tau_i$. Then the nets $(\tau_i^{+})_i$ and $(\tau_i^-)_i$ are ultranets in $T_{\leq 1}(A)$ and hence have limits $\rho_+$ and $\rho_-$ respectively in $T_{\leq 1}(A)$. Note that $(\tau_i)_i$ converges to $\tau := \rho_+- \rho_-$.
    As $f$ is continuous on $T_{\leq 1}(A)$ we get
    \begin{equation}
        \tilde f(\tau_i) = f(\tau_i^+) - f(\tau_i^-) \to f(\rho_+) - f(\rho_-) = f(\tau)
    \end{equation}
    so $\tilde f$ is continuous.

    It is obvious that that the bijections from \eqref{eq:Aff} are continuous and isomorphisms of ordered vector spaces. To show that these are homeomorphisms, it suffices to show that their composition is a homeomorphism. Let $(f_i)_i$ be a net in $T_{\mathbb R}^{\mathrm{b}}(A)^\ast$ such that $(f_i|_{T_{\leq 1}(A)})_i$ converges uniformly to zero. Let $S\subseteq T_{\mathbb R}^{\mathrm{b}}(A)$ be a compact set. Arguing as in Lemma \ref{l:compact}\footnote{Since the weak$^\ast$-topology on $T_{\mathbb R}^{\mathrm{b}}(A)$ coming from $\Ped(A)^\ast$ and from $A^\ast$ differ, we cannot simply apply Lemma \ref{l:compact}, but the same argument works.} we pick $M>0$ such that $\sup_{\tau\in S} \| \tau\| < M$. Hence for every $\tau \in S$ with Jordan decomposition $(\tau^+, \tau^-)$ we have $\tfrac{1}{M} \tau^\pm \in T_{\leq 1}(A)$. Hence
    \begin{equation}
        \sup_{\tau\in S} |f_i(\tau)| \leq M \sup_{\tau\in S}  (|f_i(\tfrac{1}{M}(\tau^+))| + |f_i(\tfrac{1}{M}(\tau^-))|) \to 0
    \end{equation}
    and therefore $T_{\mathbb R}^{\mathrm{b}}(A)^\ast \to \Aff_0 T_{\leq 1}(A)^\ast$ is a homeomorphism.

    Assume now that $T_1(A)$ is also compact. The map $T_{\mathbb R}^{\mathrm{b}}(A)^\ast \to \Aff T_1(A)$ is (as for $\Aff_0 T_{\leq 1}(A)$) injective. For surjectivity we may (again as above) let $f\in \Aff T_1(A)$ and extend it to $\tilde f\colon T_{\mathbb R}^{\mathrm{b}}(A) \to \mathbb R$. To show that $\tilde f$ is continuous it suffices by the previous part to show that $\tilde f$ is continuous on $T_{\leq 1}(A)$.

    Let $(\tau_i)_i$ be an ultranet in $T_{\leq 1}(A)$ with limit $\tau$. By the definition of ultranets, either $(\tau_i)_i$ is eventually equal to zero, or eventually contained in $T_{\leq 1}(A) \setminus \{0\}$. If $(\tau_i)_i$ is eventually zero then clearly $\lim_i \tilde f(\tau_i) = 0$ so we may assume that each $\tau_i \neq 0$. As $T_1(A)$ is compact the ultranet $(\tfrac{1}{\| \tau_i\|} \tau_i)_i$ has a limit $\rho \in T_1(A)$. Similarly, let $M = \lim_i \| \tau_i\|$. Then $\tau = M \rho$ and since $f$ is continuous on $T_1(A)$ we get
    \[
        \lim_i \tilde f(\tau_i) = \lim_i \| \tau_i\| f(\tfrac{1}{\| \tau_i\|} \tau_i) = M f(\rho) = \tilde f(\tau).
    \]
    Hence $\tilde f$ is continuous on $T_{\leq 1}(A)$ and thus continuous.

    Finally, for every $f\in T_{\mathbb R}^{\mathrm{b}}(A)^\ast$ we clearly have
    \begin{equation}
        \sup_{\tau\in T_{\leq 1}(A)} |f(\tau)| = \sup_{\tau\in T_1(A)}|f(\tau)|
    \end{equation}
    and therefore $T_{\mathbb R}^{\mathrm{b}}(A)^\ast \to \Aff T_1(A)$ is an isomorphism of ordered topological vector spaces.
\end{proof}

When we consider unbounded traces, it becomes less clear how to recover the information of the entire self-adjoint trace space $T_{\mathbb R}(A)$ from the positive trace space $T_+(A)$. We consider $\Aff_0 T_+(A)$ as an ordered topological vector space with the compact-open topology coming from $T_+(A)$.

\begin{proposition}\label{p:TAffinj}
    Let $A$ be a $\cs$-algebra. Restriction from $T_{\mathbb R}(A)$ to $T_+(A)$ defines an embedding 
    \begin{equation}
        T_{\mathbb R}(A)^\ast \to \Aff_0 T_+(A)
    \end{equation}
    of ordered topological vector spaces. 
\end{proposition}
\begin{proof}
    The map is clearly continuous and maps the positive cone to the positive cone, and it is injective by Jordan decompositions (as in the proof of Proposition \ref{p:Affbounded}). It remains to show that the inverse map on the image is continuous.

    Let $(f_i)_i$ be a net in $T_{\mathbb R}(A)^\ast$ which converges to $0$ uniformly on compact subsets of $T_+(A)$. Let $S\subseteq T_{\mathbb R}(A)$ be compact, and let $S^\pm = \{ \tau^\pm : \tau\in S\}$. For every $\cs$-subalgebra $D\subseteq \Ped(A)$ we have $DAD \subseteq \Ped(A)$ by Proposition \ref{p:upwards}. Hence by Lemma \ref{l:compact} we have
    \begin{equation}
        \sup_{\tau\in S} \| \tau^\pm|_D\| = \sup_{\tau\in S} \|\tau^\pm|_{DAD}\| \leq \sup_{\tau\in S} \| \tau|_{DAD}\| < \infty.
    \end{equation}
    Therefore, the sets $S^+$ and $S^-$ are relatively compact by Lemma \ref{l:compact}, and thus $(f_i)_i$ converges uniformly to zero on both sets. Hence
    \begin{equation}
        \sup_{\tau\in S} |f_i(\tau)| \leq \sup_{\tau\in S}( |f_i(\tau^+)| + |f_i(\tau^-)|)  \to 0. \qedhere
    \end{equation}
\end{proof}

Notice that by Theorem \ref{t:PedTcont}, surjectivity of the embedding $T_{\mathbb R}(A)^\ast \to \Aff_0 T_+(A)$ (and thus that it is an isomorphism) is equivalent to a positive answer to the Trace Question (Question \ref{q:traces}) for $A$. 
The Trace Question is therefore essentially a question of whether the information of $T_{\mathbb R}(A)$ can be recovered from $T_+(A)$.

We have a positive answer to the Trace Question for a large class of $\cs$-algebras. First note that whenever $\Ped(A) = A$, the map $T_{\mathbb R}(A)^\ast \to \Aff_0 T_+(A)$ is surjective by Proposition \ref{p:Affbounded}. Moreover, the answer depends only on the ordered topological vector space structure of $T_{\mathbb R}(A)$, so by Corollary \ref{c:Tfullher} it is stable under Morita equivalence. It is well-known (and elementary to prove) that the primitive ideal space of a $\cs$-algebra $A$ is compact if and only if $A$ contains a full hereditary $\cs$-subalgebra $B$ such that $\Ped(B) = B$, so this already affirms the question for any $\cs$-algebra with compact primitive ideal space. 

We extend this positive answer to any $\cs$-algebra $A$ which exhibits compactness more generally at the level of its positive trace space. A \textit{compact base} for the cone $T_+(A)$ of positive traces is a compact subset $C \subseteq T_+(A)$ not containing $0$ such that for each non-zero trace $\tau \in T_+(A)$ there is a unique $\lambda \in (0,\infty)$ with $\lambda \tau \in C$. The cone $T_+(A)$ has a compact base if and only if it is locally compact \cite[Theorem II.2.6]{Alfsen-book}\footnote{A cone $P$ is said to be proper if $P \cap (-P) = \{0\}$.}.

\begin{proposition}\label{p:compact base}
Let $A$ be a $\cs$-algebra whose cone $T_+(A)$ of positive traces has a compact base. Then the restriction map
 \begin{equation*}
        T_{\mathbb R}(A)^\ast \to \Aff_0 T_+(A)
 \end{equation*}
is an isomorphism of ordered topological vector spaces.
\begin{proof}
By Proposition \ref{p:TAffinj} it suffices to show the map is surjective. 
Let $f \in \Aff_0 T_+(A)$ and let $\tilde f \colon T_{\mathbb R}(A) \to \mathbb R$ be its unique linear extension. We must show that $\tilde f$ is weak$^*$-continuous. Let $C$ be a compact base for $T_+(A)$. As $C$ does not contain $0$, the open sets $\{ \tau \in C \mid \tau(e) > 1 \}$ associated to non-zero positive elements $e \in \Ped(A)$ cover $C$. By compactness and by taking a finite sum there is a positive element $e \in \Ped(A)$ such that $\tau(e) \geq 1$ for all $\tau \in C$. Set $B = \overline{eAe}$ and let $\Phi \colon T_{\mathbb R}(A) \to T^{\mathrm{b}}_{\mathbb R}(B)$ denote the restriction. Because $e\in B$ and $C$ is a base, the only positive trace in the kernel of $\Phi$ is $0$. As $\Phi$ is a lattice homomorphism by Corollary \ref{c:herlattice}, it must be injective on $T_+(A)$ and by linearity also on $T_{\mathbb R}(A)$.

Let $\bar f \colon \Phi(T_{\mathbb R}(A)) \to \mathbb R$ be the induced linear functional $\bar f (\tau|_B) = \tilde f (\tau)$ for $\tau\in T_{\mathbb R}(A)$. We claim that $\bar f$ is continuous and thus so is $\tilde f = \bar f \circ \Phi$. We show the equivalent statement that $\ker \bar f = \Phi(\ker \tilde f)$ is weak$^\ast$-closed in $T_{\mathbb R}^{\mathrm{b}}(B)$. 
By Krein--Smulian it suffices to show that the unit ball of $\ker \bar f$ is weak$^\ast$-closed. Let $(\tau_i)_i$ be an ultranet in $\ker \tilde f$ with $ \lVert \tau_i \restriction_B \rVert \leq 1$, and let $\tau_B \in T_{\mathbb R}^{\mathrm{b}}(B)$ be its limit. Taking Jordan decompositions $\tau_i^\pm$ of $\tau_i$, there are $\lambda_i^\pm \in [0,\infty)$ and $\rho_i^\pm \in C$ with $\lambda_i^\pm \rho_i^\pm = \tau_i^\pm$ and we have $\tau_i^\pm \restriction_B \in T_{\leq 1}(B)$. It follows that $\lambda_i^\pm \leq \lVert e \rVert$ for each $i$, and so the ultranets $(\lambda_i^\pm)_i$ have limits $\lambda^\pm$. The ultranets $(\rho_i^\pm)_i$ also have limits $\rho^\pm \in C$ by compactness and so $\tau_i^\pm \to \lambda^\pm \rho^\pm \in T_+(A)$. As $f$ is continuous we have 
\[
0 = \lim_i\tilde f(\tau_i) = \lim_i( \lambda_i^+ f(\rho_i^+) - \lambda_i^- f(\rho_i^-)) = \lambda^+ f(\rho^+) - \lambda^- f(\rho^-).
\] 
In particular, we have $\tau = \lambda^+ \rho^+ - \lambda^- \rho^- \in \ker \tilde f$ with $\tau_i \to \tau$, and so $\tau_B = \tau \restriction_B \in \ker \bar f$.
\end{proof}
\end{proposition}

Proposition \ref{p:compact base} covers the situation where all traces are bounded and the tracial state space is compact.

We moreover affirm Question \ref{q:traces} for a class of $\cs$-algebras containing every commutative $\cs$-algebra. These $\cs$-algebras satisfy a strong version of Lemma \ref{l:Pedher} with closed ideals instead of hereditary $\cs$-subalgebras.

\begin{proposition}\label{p:TAff}
    Suppose that $A$ is a $\cs$-algebra with the following property: for all $a_1,\dots, a_n \in \Ped(A)$, the two-sided closed ideal $J\subseteq A$ generated by $a_1,\dots,a_n$ is contained in $\Ped(A)$. Then the restriction map 
    \begin{equation}\label{eq:TRAff}
        T_{\mathbb R}(A)^\ast \to \Aff_0 T_+(A)
    \end{equation}
    is an isomorphism of ordered topological vector spaces.
\end{proposition}
\begin{proof}
    By Proposition \ref{p:TAffinj} it suffices to show the map is surjective. Let $f\in \Aff_0 T_+(A)$ and extend it to a linear functional $\tilde f \colon T_{\mathbb R}(A) \to \mathbb R$ which is weak$^\ast$-continuous on $T_+(A)$. Hence we may pick $a_1,\dots, a_n \in \Ped(A)$ such that 
   \begin{equation}
       \{ \tau \in T_+(A) : \max_{1 \leq i \leq n} |\tau(a_i)| < 1 \} \subseteq \tilde f^{-1}(\mathcal B_1(0))
   \end{equation}
   where $\mathcal B_1(0)$ is the open unit ball in $\mathbb C$. Let $J$ be the two-sided closed ideal generated by $a_1, \dots, a_n$. By assumption, $J\subseteq \Ped(A)$, and thus $\tau|_J$ is bounded for every $\tau\in T_{\mathbb R}(A)$. Consequently, we get that 
   \begin{equation}\label{eq:fboundedT_+}
       |\tilde f(\tau)| \leq \max_{1\leq i \leq n} |\tau(a_i)| \leq \| \tau|_J\| \max_{1\leq i \leq n} \| a_i\| = \| \tau|_J\| M
   \end{equation}
   for all $\tau\in T_+(A)$, where  $M := \max_{1\leq i \leq n} \| a_i\|$.

   For any $\tau\in T_{\mathbb R}(A)$ with Jordan decomposition $(\tau^+, \tau^-)$, we have by Proposition \ref{p:JordanPed} that $(\tau^+|_J, \tau^-|_J)$ is the Jordan decomposition of $\tau|_J$ and thus
   \begin{equation}\label{eq:fbounded}
       |\tilde f(\tau)| \leq |\tilde f(\tau^+)| + |\tilde f(\tau^-)| \stackrel{\eqref{eq:fboundedT_+}}{\leq} (\| \tau^+|_J\| + \| \tau^-|_J\|) M = \| \tau|_J\| M.
   \end{equation}
   Note that every bounded trace $\rho$ on $J$ extends canonically to a trace $\tilde \rho \in T_{\mathbb R}(A)$ with $\| \tilde \rho \| = \| \rho\|$ by composing the trace $\rho^{\ast \ast}$ on $J^{\ast \ast}$ with the canonical $\ast$-homomorphism $A \to J^{\ast \ast}$.  
    In particular, $\tilde f$ induces a linear functional $\overline f \colon T_{\mathbb R}^{\mathrm{b}} (J) \to \mathbb R$ by
    \begin{equation}\label{restriction to J equation}
        \overline f(\tau|_J) = \tilde f(\tau) , \qquad \tau \in T_{\mathbb R}(A).
    \end{equation}
    This is well-defined by \eqref{eq:fbounded}. To check that this linear functional is continuous, it suffices by Proposition \ref{p:Affbounded} to show that it is continuous on $T_{\leq 1}(J)$. Let $(\rho_i)_i$ be an ultranet in $T_{\leq 1}(J)$ with limit $\rho\in T_{\leq 1}(J)$. Let $\tilde \rho_i \in T_{\leq 1}(A)$ be the canonical extensions of $\rho_i$. Since $T_{\leq 1}(A)$ is weak$^\ast$-compact, let $\tau \in T_{\leq 1}(A)$ be the limit of the ultranet $(\tilde \rho_i)_i$. Note that $\tau$ might not be the canonical extension of $\rho$, but that it is an extension of $\rho$ nonetheless. As $\tilde f$ is continuous on $T_+(A)$ we get
    \begin{equation}
        \lim_i \overline f(\rho_i)  = \lim_i \tilde f(\tilde \rho_i) = \tilde f(\tau) = \overline f(\rho).
    \end{equation}
    Hence $\overline f$ is continuous. By continuity of the restriction to $J$ and \eqref{restriction to J equation}, $\tilde f$ is continuous.
\end{proof}

In summary we get the following. 

\begin{corollary}\label{c:traceproblem}
The class of $\cs$-algebras $A$ for which the Trace Question has a positive answer, or equivalently the canonical map 
 \begin{equation*}
        T_{\mathbb R}(A)^\ast \to \Aff_0 T_+(A)
 \end{equation*}
is an isomorphism of ordered topological vector spaces, is closed under Morita equivalence, and contains every $\cs$-algebra $A$ satisfying one of the following:
    \begin{itemize}
        \item[(a)] all the tracial weights on $A$ are bounded and $T_1(A)$ is compact; or
        \item[(b)] the primitive ideal space of $A$ is compact (e.g.~if $A$ is simple or contains a full projection); or more generally
        \item[(c)] the cone $T_+(A)$ of positive traces has a compact base (equivalently, it is locally compact); or
        \item[(d)] $A$ is commutative.
    \end{itemize}
\end{corollary}

\section{Traces on groupoid $\cs$-algebras}\label{s:groupoid}

An étale groupoid $\mathcal G$ is a topological groupoid such that the range map $r \colon \mathcal G \to \mathcal G^0$ (and thus also the source map) is a local homeomorphism, and we assume that the unit space $\mathcal G^0$ is locally compact and Hausdorff. By a \textit{twist} over $\mathcal G$ we mean a Fell line bundle $\mathcal L \to \mathcal G$, and we call the pair $(\mathcal G, \mathcal L)$ a \textit{twisted étale groupoid}. 
For an open bisection $U \subseteq \mathcal G$, we write $\Gamma_c(U, \mathcal L)$ for the space of compactly supported continuous sections $U \to \mathcal L$. We include $\Gamma_c(U, \mathcal L)$ into the space of bounded sections $\mathcal G \to \mathcal L$ by setting the value to be $0$ outside $U$, and set $\Gamma_c(\mathcal G, \mathcal L)$ to be the span of all these sections over open bisections $U$. When $\mathcal G$ is Hausdorff, this is the space of compactly supported sections, but for non-Hausdorff groupoids we warn the reader that these sections are not necessarily continuous. We identify the space of sections $\Gamma_c(\mathcal G^0, \mathcal L)$ supported on the unit space with $C_c(\mathcal G^0)$ as the restriction of $\mathcal L$ to $\mathcal G^0$ is the trivial bundle whose fibres have the structure of the $\cs$-algebra $\mathbb C$. We refer to \cite{ExelPitts} for an introduction to twisted étale groupoids from the line bundle perspective and their $\cs$-algebras. 

Convolution defines a $*$-algebra structure on $\Gamma_c(\mathcal G, \mathcal L)$. Given a faithful $*$-repre\-sent\-ation $\pi \colon \Gamma_c(\mathcal G, \mathcal L) \to \mathcal B(\mathcal H)$ we write $\|\cdot\|_\pi$ for the induced $\cs$-norm on $\Gamma_c(\mathcal G, \mathcal L)$ and we denote the completion by $\cs_\pi(\mathcal G, \mathcal L)$. For the trivial Fell line bundle $\mathcal L = \mathbb C \times \mathcal G \to \mathcal G$ we recover the usual untwisted $*$-algebras and write $C_c(\mathcal G)$ and $\cs_\pi(\mathcal G)$. The reduced $\cs$-algebra $\cs_\lambda(\mathcal G, \mathcal L)$ is the $\cs$-algebra defined by the left regular representation $\lambda \colon \Gamma_c(\mathcal G, \mathcal L) \to \mathcal B(\bigoplus_{x \in G^0} \ell^2(\mathcal G_x, \mathcal L))$.

\begin{lemma}
   Let $(\mathcal G, \mathcal L)$ be a twisted étale groupoid and $\pi \colon \Gamma_c(\mathcal G, \mathcal L) \to \mathcal B(\mathcal H)$ a faithful $\ast$-representation. Then $\Gamma_c(\mathcal G, \mathcal L) \subseteq \Ped(\cs_\pi(\mathcal G, \mathcal L))$.
\end{lemma}

\begin{proof}
    Let $f\in \Gamma_c(\mathcal G, \mathcal L)$. Then there is a compact set $Y \subseteq \mathcal G^0$ such that $f(g) = 0$ for any $g \in \mathcal G$ with $s(g) \in \mathcal G^0 \setminus Y$, and we may find $g\in C_c(\mathcal G^0)$ such that $g(y) = 1$ for all $y\in Y$ (see for instance \cite[Proposition 1.7.5]{Pedersen-book-analysisnow}). Hence $g\in C_c(\mathcal G^0) \subseteq \Ped(\cs_\pi(\mathcal G, \mathcal L))$ and $f\ast g = f$. As $\Ped(\cs_\pi(\mathcal G, \mathcal L))$ is an ideal it follows that $f\in \Ped(\cs_\pi(\mathcal G, \mathcal L))$.
\end{proof}

If $(\mathcal G, \mathcal L)$ is a twisted étale groupoid and $Y \subseteq \mathcal G^0$ is open, we let $\Gamma_c(\mathcal G, \mathcal L)|_Y$ denote the $\ast$-subalgebra of all $f \in \Gamma_c(\mathcal G, \mathcal L)$ for which $f = 0$ outside of $\mathcal G^Y_Y := \{g \in \mathcal G : r(g),s(g) \in Y \}$. The following describes $T_{\mathbb R}(\cs _\pi(\mathcal G, \mathcal L))$ as an ordered vector space in terms of traces on $\Gamma_c(\mathcal G, \mathcal L)$. We note that one could also describe the topology using Proposition \ref{p:extendtraces}, but that the topology is not obviously natural in this case.

\begin{proposition}\label{traces on groupoid algebras}
Let $(\mathcal G, \mathcal L)$ be a twisted étale groupoid and $\pi \colon \Gamma_c(\mathcal G, \mathcal L) \to \mathcal B(\mathcal H)$ a faithful $\ast$-representation. Restriction defines a one-to-one correspondence between self-adjoint/positive continuous traces on $\Ped(\cs_\pi(\mathcal G, \mathcal L))$ and self-adjoint/positive traces $\tau$ on $\Gamma_c(\mathcal G, \mathcal L)$ which are $\| \cdot\|_\pi$-bounded on $\Gamma_c(\mathcal G, \mathcal L)|_Y$ for each relatively compact open $Y \subseteq \mathcal G^0$. 
\end{proposition}

\begin{proof}
In Proposition \ref{p:extendtraces}, take $A = \cs_\pi(\mathcal G, \mathcal L)$, $A_0 = \Gamma_c(\mathcal G, \mathcal L)$ and $X = C_c(\mathcal G^0)$. It then suffices to show that a tracial functional $\tau \colon \Gamma_c(\mathcal G, \mathcal L) \to \mathbb C$ is $\| \cdot \|_\pi$-bounded on $\Gamma_c(\mathcal G, \mathcal L)|_Y$ for all relatively compact open $Y\subseteq \mathcal G^0$ if and only if $\tau$ is $\| \cdot \|_\pi$-bounded on $f \Gamma_c(\mathcal G, \mathcal L) f^\ast$ for all $f\in C_c(\mathcal G^0)$. If $Y\subseteq \mathcal G^0$ is relatively compact and open and $f \in C_c(\mathcal G^0)$ is a function which is $1$ on $Y$, then $\Gamma_c(\mathcal G, \mathcal L)|_Y \subseteq f\Gamma_c(\mathcal G, \mathcal L)f^\ast$ and thus the ``if''-part follows. Similarly, if $f\in C_c(\mathcal G^0)$ then $\mathrm{supp}(f)$ is contained in a relatively compact open $Y \subseteq \mathcal G^0$ and $f\Gamma_c(\mathcal G, \mathcal L)f^\ast \subseteq \Gamma_c(\mathcal G, \mathcal L)|_Y$, hence the ``only if''-part follows.
\end{proof}

Traces on groupoid $\cs$-algebras arise naturally from invariant real Radon measures on the unit space. Recall that we let $\mathrm{Rad}_{\mathbb R}(X)$ denote the ordered topological vector space of real Radon measures on the locally compact space $X$.  

\begin{definition}
Let $\mathcal G$ be an étale groupoid. A real Radon measure $\mu$ on $\mathcal G^0$ is \textit{$\mathcal G$-invariant} if for every open bisection $U \subseteq \mathcal G$ and every $f \in C_c(r(U))$, we have
\[ \int_{r(U)} f \, \mathrm d \mu =   \int_{s(U)} f \circ \alpha_U \, \mathrm d \mu, \]
where $\alpha_U \colon s(U) \to r(U)$ is the homeomorphism sending $s(g)$ to $r(g)$ for each $g \in U$. We note that this formula clearly extends to every $f \in L^1(r(U),\mu)$.

We let $\mathrm{Rad}_{\mathbb R}^\mathcal{G}(\mathcal G^0)$ denote the ordered topological vector space of real $\mathcal G$-invariant Radon measures with the weak$^\ast$-topology coming from $C_c(\mathcal G^0)$. 
\end{definition}

Recall that $C_0(\mathcal G^0)$ embeds canonically as a $\cs$-subalgebra of $\cs_\pi(\mathcal G, \mathcal L)$ and hence there is a canonical map $T_{\mathbb R}(\cs_\pi(\mathcal G, \mathcal L)) \to T_{\mathbb R}(C_0(\mathcal G^0)) \cong \mathrm{Rad}_{\mathbb R}(\mathcal G^0)$ given by restriction.

\begin{lemma}\label{l:res}
     For any twisted étale groupoid $(\mathcal G,\mathcal L)$ with faithful $\ast$-re\-pre\-sen\-tation $\pi \colon \Gamma_c(\mathcal G, \mathcal L) \to \mathcal B(\mathcal H)$ and any trace $\tau \in T_{\mathbb R}(\cs_\pi(\mathcal G, \mathcal L))$, the Radon measure on $\mathcal G^0$ obtained by restriction of $\tau$ to $C_c(\mathcal G^0)$ is $\mathcal G$-invariant.
\end{lemma}
\begin{proof}
    Let $U \subseteq \mathcal G$ be an open bisection and let $f \in C_c(r(U))$. 
    To show that $\tau(f) = \tau(f \circ \alpha_U)$, we may without loss of generality assume that $U$ is relatively compact so that $\tau$ is bounded on $C_c(r(U) \cup s(U))$. 
    Viewing $\Gamma_0(U, \mathcal L)$ as a full left Hilbert $C_0(r(U))$-module, we may for any $\epsilon > 0$ find finitely many $k_i,h_i \in \Gamma_c(U, \mathcal L)$ such that $\| f - \sum_i k_i \ast h_i^\ast \| < \epsilon$. 
    For $g \in U$ we compute\footnote{For $a \in \mathcal L_g$ and $b \in \mathcal L_{g^{-1}}$ the elements $a \cdot b$ and $b \cdot a$ represent the same complex number because if $b = \lambda a^\ast$ then both are $\lambda \|a\|^2$.}
    \[ \sum_i ( h_i^\ast \ast k_i ) (s(g)) = \sum_i h_i^\ast(g^{-1}) \cdot k_i(g) = \sum_i k_i(g) \cdot h_i^\ast(g^{-1}) = \sum_i ( k_i \ast h_i^\ast )(r(g)) \]
    and therefore $\| f \circ \alpha_U - \sum_i   h_i^\ast \ast k_i  \| < \epsilon$. Since $\epsilon > 0$ was arbitrary and $\tau$ is bounded on $C_c(r(U) \cup s(U))$, we conclude that $\tau(f) = \tau(f \circ \alpha_U)$.    
\end{proof}

For Hausdorff groupoids we have a canonical conditional expectation $E \colon \cs_\pi(\mathcal G, \mathcal L) \to C_0(\mathcal G^0)$ extending the map $\Gamma_c(\mathcal G, \mathcal L) \to C_c(\mathcal G^0)$ given by $f \mapsto f|_{\mathcal G^0}$. This is often used to induce bounded traces on $\cs_\pi(\mathcal G, \mathcal L)$ coming from bounded $\mathcal G$-invariant measures on $\mathcal G^0$. However, $E$ does not map $\Ped(\cs_\pi(\mathcal G, \mathcal L))$ into $C_c(\mathcal G^0) = \Ped(C_0(\mathcal G^0))$ in general, so this approach cannot a priori be generalised to unbounded traces. 

For instance, if $(C_0(\mathcal G^0), \cs_\pi(\mathcal G, \mathcal L)) = (C_0(\mathbb N), \mathcal K(\ell^2(\mathbb N)))$ with standard matrix units $e_{j,k}\in \mathcal K(\ell^2(\mathbb N))$, then the rank one projection $p = \sum_{j,k=1}^\infty \tfrac{\sqrt{6}}{jk \pi} e_{j,k}$ is in $\Ped(\mathcal K(\ell^2(\mathbb N)))$ (since it is finite rank), but $E(p) =  \sum_{j=1}^\infty  \tfrac{\sqrt{6}}{j^2 \pi} e_{j,j} \in C_0(\mathbb N)$ which does not have compact support. 

However, $\mathcal G$-invariant Radon measures do nonetheless give rise to traces on $\cs_\pi(\mathcal G, \mathcal L)$ even when $\mathcal G$ is not Hausdorff. 

\begin{corollary}\label{c:groupoidtrace}
Let $(\mathcal G, \mathcal L)$ be a twisted étale groupoid and let $\pi \colon \Gamma_c(\mathcal G, \mathcal L) \to \mathcal B(\mathcal H)$ be a $*$-representation which weakly contains the regular representation $\lambda$. Every $\mathcal G$-invariant real (resp.~positive) Radon measure $\mu$ on $\mathcal G^0$ induces a self-adjoint (resp.~positive) continuous trace $\tau_\mu$ on $\Ped(\cs_\pi(\mathcal G, \mathcal L))$ which is uniquely determined by
\[
\tau_\mu(f) = \int_{\mathcal G^0} f|_{\mathcal G^0} \, \mathrm d \mu, \qquad \textrm{for }f\in \Gamma_c(\mathcal G, \mathcal L).
\]
Moreover, if $\mathcal G$ is principal then every self-adjoint (resp.~positive) continuous trace on $\Ped(\cs_\pi(\mathcal G, \mathcal L))$ is of this form.
\end{corollary}

\begin{proof}
We first define $\tau_\mu$ only for $f \in \Gamma_c(\mathcal G, \mathcal L)$ by 
\[ \tau_\mu(f) =  \int_{\mathcal G^0} f|_{\mathcal G^0} \, \mathrm d \mu. \]
To check that $\tau_\mu$ is tracial on $ \Gamma_c(\mathcal G, \mathcal L)$, it suffices to consider $a,b \in \Gamma_c(\mathcal G, \mathcal L)$ supported on open bisections $U, V \subseteq \mathcal G$ respectively. 
The open bisection $W = U \cap V^{-1}$ has range $r(W) = UV \cap \mathcal G^0$ and source $s(W) = VU \cap \mathcal G^0$, and so 
$ab |_{\mathcal G^0} \in L^1(r(W), \mu)$ and $ba |_{\mathcal G^0} \in L^1(s(W),\mu)$\footnote{If $\mathcal G$ is Hausdorff then these are moreover compactly supported continuous functions.}.
For $g \in W$ we compute
\[ ab(r(g)) = a(g) \cdot b(g^{-1}) = b(g^{-1}) \cdot a(g) = ba(s(g)) \]
and thus $ab |_{\mathcal G^0} = ba |_{\mathcal G^0} \circ \alpha_W$. The $\mathcal G$-invariance of $\mu$ then implies that $\tau_\mu(ab) = \tau_\mu(ba)$. For any relatively compact open $Y \subseteq \mathcal G^0$ and any $f \in \Gamma_c(\mathcal G, \mathcal L)|_Y$, we have 
\[
\lvert \tau_\mu(f) \rvert \leq \lVert f |_{\mathcal G^0} \rVert_\infty \lvert\mu\rvert(Y) \leq \lVert f \rVert_\lambda \lvert\mu\rvert(Y) \leq \lVert f \rVert_\pi \lvert\mu\rvert(Y),
\]
and so by Proposition \ref{traces on groupoid algebras} $\tau_\mu$ extends uniquely to a continuous trace on $\Ped(\cs_\pi(\mathcal G, \mathcal L))$.

Now suppose $\mathcal G$ is principal and let $\tau \in T_{\mathbb R}(\cs_\pi(\mathcal G, \mathcal L))$. Then the composition of $\tau$ with the inclusion $C_c(\mathcal G^0) \hookrightarrow \Gamma_c(\mathcal G, \mathcal L)$ defines a $\mathcal G$-invariant Radon measure $\mu$ on $\mathcal G^0$ by Lemma \ref{l:res}. By principality we may cover $\mathcal G$ with 
open bisections $U \subseteq \mathcal G$ such that either $U \subseteq \mathcal G^0$ or $r(U) \cap s(U) = \emptyset$. Thus, by Proposition \ref{traces on groupoid algebras}, to check $\tau = \tau_\mu$ it suffices by linearity to check that $\tau(f) = \tau_\mu(f)$ for any $f \in \Gamma_c(\mathcal G, \mathcal L)$ whose support is contained in such a bisection $U$. This is immediate for $U \subseteq \mathcal G^0$. If instead $r(U) \cap s(U) = \emptyset$, take $h \in C_c(\mathcal G^0)$ which is $1$ on $s(\mathrm{supp}(f))$ and $0$ on $r(\mathrm{supp}(f))$. Then $f \ast h = f$ and $h \ast f = 0$ so $\tau(f) = 0 = \tau_\mu(f)$.
\end{proof}

\begin{definition}
    Let $(\mathcal G, \mathcal L)$ be a twisted 
    étale groupoid and let $\pi \colon \Gamma_c(\mathcal G, \mathcal L) \to \mathcal B(\mathcal H)$ be a $*$-representation which weakly contains $\lambda$. We let $T_{\mathbb R}^{\mathcal G}(\cs_\pi(\mathcal G, \mathcal L)) \subseteq T_{\mathbb R}(\cs_\pi(\mathcal G, \mathcal L))$ denote the subspace of traces induced by real $\mathcal G$-invariant Radon measures on $\mathcal G^0$. 
\end{definition}

\begin{remark}
If $\mathcal G$ is non-Hausdorff, the induced traces on $\cs_\lambda(\mathcal G, \mathcal L)$ may not be compatible with the quotient $\cs_{\mathrm{ess}}(\mathcal G, \mathcal L)$ of $\cs_\lambda(\mathcal G, \mathcal L)$ known as the essential $\cs$-algebra (see \cite{KwasMeyer} for details). 

For example, if $\mathcal G$ is the `interval with two origins', which has unit space $[0,1]$ and only one non-trivial arrow $g = g^{-1}$ with $s(g) = r(g) = 0$, the essential $\cs$-algebra is simply $C([0,1])$. The reduced algebra $\cs_\lambda(\mathcal G)$ is isomorphic to $\mathbb C \oplus C([0,1])$ with $C(\mathcal G^0) = C([0,1])$ sitting inside $\mathbb C \oplus C([0,1])$ as the subalgebra $\{ (f(0), f) : f \in C([0,1])\}$. When $\mu$ is the point mass at the origin, the induced trace on $\cs_\lambda(\mathcal G)$ is $\tau_\mu(z, f) = \tfrac{1}{2}(z + f(0)) $ for $z\in \mathbb C$ and $f\in C([0,1])$. 
\end{remark}

For Hausdorff groupoids we can describe the structure of $T_{\mathbb R}^\mathcal G(\cs_\pi(\mathcal G, \mathcal L))$ more precisely. The following lemma, which makes use of the canonical conditional expectation, is the key.

\begin{lemma}\label{l:grpoidtop}
    Let $(\mathcal G, \mathcal L)$ be a twisted Hausdorff étale groupoid with a $*$-representation $\pi \colon \Gamma_c(\mathcal G, \mathcal L) \to \mathcal B(\mathcal H)$ which weakly contains the regular representation $\lambda$, and let $a\in \Ped(\cs_\pi(\mathcal G, \mathcal L))_+$. Then there exists a positive element $f\in C_c(\mathcal G^0)$ such that $\tau_\mu(a) = \int_{\mathcal G^0} f \mathrm d\mu$ for every $\mathcal G$-invariant real Radon measure $\mu$ on $\mathcal G^0$. 
\end{lemma}
\begin{proof}
    As $C_c(\mathcal G^0)$ generates $\cs_\pi(\mathcal G, \mathcal L)$ as a two-sided closed ideal, we may by Lemma \ref{l:Xfull}  find $m\in \mathbb N$, $g_1,\dots, g_m \in C_c(\mathcal G^0)$ and $c_1,\dots, c_m \in \cs_\pi(\mathcal G, \mathcal L)$ such that $a = \sum_{j=1}^m c_j^\ast g_j^\ast g_j c_j$. By considering each term in this sum individually, we may assume without loss of generality that $a = c^\ast g^\ast g c$ for some $g\in C_c(\mathcal G^0)$ and $c \in \cs_\pi(\mathcal G, \mathcal L)$. 

    Let $(h_n)_n$ be a sequence in $\Gamma_c(\mathcal G, \mathcal L)$ converging to $c$. Let $E\colon \cs_\pi(\mathcal G, \mathcal L) \to C_0(\mathcal G^0)$ be the conditional expectation which extends the restriction map $\Gamma_c(\mathcal G, \mathcal L) \to C_c(\mathcal G^0)$, $f\mapsto f|_{\mathcal G^0}$. Define $f := g E(cc^\ast) g^\ast \in C_c(\mathcal G^0)$. Since $E$ is a bounded $C_0(\mathcal G^0)$-bimodule map, we have
    \begin{equation}
        \lim_n (g h_n h_n^\ast g^\ast)|_{\mathcal G^0} = \lim_n E(gh_n h_n^\ast g^\ast) =  E(gcc^\ast g^\ast ) = g E(cc^\ast) g^\ast = f.
    \end{equation} 
    Now, let $\mu\in \mathrm{Rad}_{\mathbb R}^{\mathcal G}(\mathcal G^0)$. As $\tau_\mu$ is continuous, and therefore bounded on $\overline{g \cs_\pi(\mathcal G, \mathcal L) g^\ast}$, we have (by the dominated convergence theorem\footnote{One writes $\mu$ as a linear combination of positive Radon measures and uses that $(gh_n h_n^\ast g^\ast)|_{\mathcal G^0}$ is bounded by the function $\sup_{n\in \mathbb N} \| h_n\|^2 gg^\ast$ which is in $C_c(\mathcal G^0)$ and is therefore integrable for all positive Radon measures on $\mathcal G^0$.})
    \begin{equation}
        \tau_\mu(a) = \tau_\mu(gcc^\ast g^\ast) = \lim_n \tau_\mu(gh_n h_n^\ast g^\ast) = \lim_n \int_{\mathcal G^0} (g h_n h_n^\ast g^\ast)|_{\mathcal G^0} \mathrm d\mu = \int_{\mathcal G^0} f \mathrm d \mu. 
    \end{equation}
\end{proof}

\begin{corollary}\label{c:principal}
    Let $(\mathcal G, \mathcal L)$ be a twisted Hausdorff étale groupoid and let $\pi \colon \Gamma_c(\mathcal G, \mathcal L) \to \mathcal B(\mathcal H)$ be a $*$-representation which weakly contains the regular representation $\lambda$. Then the map $\mathrm{Rad}_{\mathbb R}^{\mathcal G}(\mathcal G^0) \to T_{\mathbb R}^{\mathcal G}(\cs_\pi(\mathcal G, \mathcal L))$ is an isomorphism of ordered topological vector spaces.

    In particular, for any principal twisted étale groupoid $(\mathcal G, \mathcal L)$ with a $*$-representation $\pi$ weakly containing the regular representation $\lambda$, the restriction map $T_{\mathbb R}(\cs_\pi(\mathcal G, \mathcal L)) \to \mathrm{Rad}_{\mathbb R}^{\mathcal G}(\mathcal G^0)$ is an isomorphism of ordered topological vector spaces.
\end{corollary}
\begin{proof}
    The map is surjective by definition, and injective since $\tau_\mu|_{C_c(\mathcal G^0)} = \mu$. The positive cones are obviously preserved, and the map is a homeomorphism by Lemma \ref{l:grpoidtop}.

    The final part follows from Corollary \ref{c:groupoidtrace} since principal groupoids are automatically Hausdorff.\footnote{If $g_1, g_2 \in \mathcal G$ are distinct, principality implies that either $s(g_1) \ne s(g_2)$ or $r(g_1) \ne r(g_2)$, after which Hausdorffness of $\mathcal G^0$ can be used to separate $g_1$ and $g_2$ by open sets in $\mathcal G$.}
\end{proof}

\newcommand{\etalchar}[1]{$^{#1}$}


\begin{thebibliography}{EGLN20}

\bibitem[AS21]{AfsarSims}
Zahra Afsar and Aidan Sims.
\newblock K{MS} states on the {$\cs$}-algebras of {F}ell bundles over
  groupoids.
\newblock {\em Math. Proc. Cambridge Philos. Soc.}, 170(2):221--246, 2021.

\bibitem[Alf71]{Alfsen-book}
Erik~M. Alfsen.
\newblock {\em Compact convex sets and boundary integrals}, volume Band 57 of
  {\em Ergebnisse der Mathematik und ihrer Grenzgebiete [Results in Mathematics
  and Related Areas]}.
\newblock Springer--Verlag, New York--Heidelberg, 1971.

\bibitem[BR24]{BRtraces}
Bruce Blackadar and Mikael Rørdam.
\newblock The space of tracial states on a $\cs$-algebra.
\newblock Preprint, arXiv:2409.09644, 2024.

\bibitem[Bou52]{Bourbaki-XIII}
Nicolas Bourbaki.
\newblock {\em El\'{e}ments de math\'{e}matique. {XIII}. {P}remi\`ere partie:
  {L}es structures fondamentales de l'analyse. {L}ivre {VI}: {I}nt\'{e}gration.
  {C}hapitre {I}: {I}n\'{e}galit\'{e}s de convexit\'{e}. {C}hapitre {II}:
  {E}spaces de {R}iesz. {C}hapitre {III}: {M}esures sur les espaces localement
  compacts. {C}hapitre {IV}: {P}rolongement d'une mesure; espaces {$L^p$}}.
\newblock Actualit\'{e}s Scientifiques et Industrielles [Current Scientific and
  Industrial Topics], No. 1175. Hermann \& Cie, Paris, 1952.



\bibitem[BO08]{BrownOzawa-book-approx}
Nathanial~P. Brown and Narutaka Ozawa.
\newblock {\em {$\cs$}-algebras and finite-dimensional approximations},
  volume~88 of {\em Graduate Studies in Mathematics}.
\newblock American Mathematical Society, Providence, RI, 2008.


\bibitem[CGS{\etalchar{+}}]{CGSTW2}
José Carrión, James Gabe, Christopher Schafhauser, Aaron Tikuisis, and Stuart
  White.
\newblock Classifying {$^\ast$}-homomorphisms {II}.
\newblock {\em In preparation}.

\bibitem[Chr23]{ChristensenKMSGroupoid}
Johannes Christensen.
\newblock The structure of {KMS} weights on \'etale groupoid {$\cs$}-algebras.
\newblock {\em J. Noncommut. Geom.}, 17(2):663--691, 2023.

\bibitem[EGLN20]{EGLN}
George~A. Elliott, Guihua Gong, Huaxin Lin, and Zhuang Niu.
\newblock The classification of simple separable {KK}-contractible {$\cs$}-algebras with finite nuclear dimension.
\newblock {\em J. Geom. Phys.}, 158:103861, 51, 2020.

\bibitem[ERS11]{ERS}
George~A. Elliott, Leonel Robert, and Luis Santiago.
\newblock The cone of lower semicontinuous traces on a {$\cs$}-algebra.
\newblock {\em Amer. J. Math.}, 133(4):969--1005, 2011.

\bibitem[EP22]{ExelPitts}
Ruy Exel and David~R. Pitts.
\newblock {\em Characterizing groupoid {$\cs$}-algebras of non-{H}ausdorff
  \'etale groupoids}, volume 2306 of {\em Lecture Notes in Mathematics}.
\newblock Springer, Cham, [2022] \copyright2022.

\bibitem[Gab20]{Gabe-O2class}
James Gabe.
\newblock A new proof of {K}irchberg's {$\mathcal O_2$}-stable classification.
\newblock {\em J. Reine Angew. Math.}, 761:247--289, 2020.

\bibitem[GL20]{GL2}
Guihua Gong and Huaxin Lin.
\newblock On classification of non-unital amenable simple {$\cs$}-algebras,
  {II}.
\newblock {\em J. Geom. Phys.}, 158:103865, 102, 2020.

\bibitem[GL22]{GL3}
Guihua Gong and Huaxin Lin.
\newblock On classification of non-unital amenable simple {$\cs$}-algebras,
  {III} : {T}he range and the reduction.
\newblock {\em Ann. K-Theory}, 7(2):279--384, 2022.

\bibitem[GL24]{GL4}
Guihua Gong and Huaxin Lin.
\newblock On classification of non-unital amenable simple {$\cs$} -algebras,
  {IV} : {S}tably projectionless {$\cs$}-algebras.
\newblock {\em Ann. K-Theory}, 9(2):143--339, 2024.

\bibitem[Kad51]{Kadison-duality}
Richard~V. Kadison.
\newblock A representation theory for commutative topological algebra.
\newblock {\em Mem. Amer. Math. Soc.}, 7:39, 1951.


\bibitem[KR02]{KirchbergRordam-absorbingOinfty}
Eberhard Kirchberg and Mikael R{\o}rdam.
\newblock Infinite non-simple {$\cs$}-algebras: absorbing the {C}untz algebras
  {$\mathcal{O}_\infty$}.
\newblock {\em Adv. Math.}, 167(2):195--264, 2002.

\bibitem[KM21]{KwasMeyer}
Bartosz~K. Kwa\'sniewski and Ralf Meyer.
\newblock Essential crossed products for inverse semigroup actions: simplicity
  and pure infiniteness.
\newblock {\em Doc. Math.}, 26:271--335, 2021.

\bibitem[LZ24]{LiZhang}
Kang Li and Jiawen Zhang.
\newblock Tracial states on groupoid $\cs$-algebras and essential freeness.
\newblock {\em J. Noncommut. Geom., to appear}, 2024.

\bibitem[LR19]{LiRenault}
Xin Li and Jean Renault.
\newblock Cartan subalgebras in {$\cs$}-algebras. {E}xistence and
  uniqueness.
\newblock {\em Trans. Amer. Math. Soc.}, 372(3):1985--2010, 2019.

\bibitem[Meg98]{Megginson}
Robert~E. Megginson.
\newblock {\em An introduction to {B}anach space theory}, volume 183 of {\em
  Graduate Texts in Mathematics}.
\newblock Springer--Verlag, New York, 1998.

\bibitem[Nes13]{NeshveyevKMSGroupoid}
Sergey Neshveyev.
\newblock K{MS} states on the {$\cs$}-algebras of non-principal groupoids.
\newblock {\em J. Operator Theory}, 70(2):513--530, 2013.

\bibitem[NS22]{NeshveyevStammeier}
Sergey Neshveyev and Nicolai Stammeier.
\newblock The groupoid approach to equilibrium states on right {LCM}
              semigroup {$\cs$}-algebras.
\newblock {\em J. Lond. Math. Soc. (2)}, 105(1):220--250, 2022.

\bibitem[Ped66]{Pedersen-MeasureI}
Gert~K. Pedersen.
\newblock Measure theory for {$\cs$}-algebras.
\newblock {\em Math. Scand.}, 19:131--145, 1966.

\bibitem[Ped68]{Pedersen-MeasureII}
Gert~K. Pedersen.
\newblock Measure theory for {$\cs$}-algebras. {II}.
\newblock {\em Math. Scand.}, 22:63--74, 1968.

\bibitem[Ped69]{Pedersen-MeasureIII}
Gert~K. Pedersen.
\newblock Measure theory for {$\cs$}-algebras. {III}.
\newblock {\em Math. Scand.}, 25:71--93, 1969.

\bibitem[Ped79]{Pedersen-book-automorphism}
Gert~K. Pedersen.
\newblock {\em {$\cs$}-algebras and their automorphism groups}, volume~14
  of {\em London Mathematical Society Monographs}.
\newblock Academic Press Inc. [Harcourt Brace Jovanovich Publishers], London,
  1979.

\bibitem[Ped89]{Pedersen-book-analysisnow}
Gert~K. Pedersen.
\newblock {\em Analysis now}, volume 118 of {\em Graduate Texts in
  Mathematics}.
\newblock Springer--Verlag, New York, 1989.

\bibitem[Rob09]{Robert}
Leonel Robert.
\newblock On the comparison of positive elements of a {$\cs$}-algebra by lower
  semicontinuous traces.
\newblock {\em Indiana Univ. Math. J.}, 58(6):2509--2515, 2009.

\end{thebibliography}
\end{document}